\documentclass[11pt]{article}
\usepackage{amsmath,amssymb}
\usepackage{cases}
\usepackage{graphicx}
\newtheorem{theorem}{Theorem}

\newtheorem{example}{Example}
\newtheorem{lemma}[theorem]{Lemma}

\newcommand{\Ref}[1]{\mbox{(\ref{#1})}}
\newcommand{\bull}{\vrule height 1.8ex width 1.0ex depth 0ex}

\newenvironment{proof}{{\bf Proof : }}{\hfill$\bull$\medskip}

\newcommand{\Dev}[2]{{\frac{\partial {#1}}{\partial {#2}}}}

\newcommand{\AEW}{\textrm{a.e.}~~}

\newcommand{\Norm}[2]{\left\|#1\right\|_{#2}}
\newcommand{\Inner}[3]{\left<#1,#2\right>_{#3}}

\newcommand{\Set}[2]{{\left\{{#1}~;~{#2}\right\}}}

\newcommand{\Cinv}{C_{\textrm{inv}}}
\title{Numerical verification of solutions for nonlinear parabolic problems}
\author{
Kouji Hashimoto$^{1}$, Takehiko Kinoshita$^{2}$, Mitsuhiro T. Nakao$^{3}$ 
}
\date{\small\sl 
$^{1}$Division of Infant Education, Nakamura Gakuen Junior College, Fukuoka 814-0198, Japan\\
$^{2}$Research Institute for Information Technology, Kyushu University, Fukuoka 819-0395, Japan\\
$^{3}$Faculty of Science and Engineering, Waseda University, Tokyo 169-8555, Japan
}
\begin{document}
\maketitle
\begin{abstract}
In this paper, we present a numerical verification method of solutions for nonlinear parabolic initial boundary value problems.
Decomposing the problem into a nonlinear part and an initial value part, we apply Nakao's projection method, which is based on the full-discrete finite element method with constructive error estimates, to the nonlinear part and use the theoretical analysis for the heat equation to the initial value part, respectively. 
We show some verified examples for solutions of nonlinear problems from initial value to the neighborhood of the stationary solutions, which confirm us the actual effectiveness of our method.
\end{abstract}
\section{Introduction}
Let $\nu$ be a positive constant, and let $\Omega \subset \mathbb{R}^d$ $(d \in \{1,2,3\})$ be a bounded polygonal or polyhedral domain, $J:=(0,T) $ a bounded open interval. Then we consider the following nonlinear parabolic initial boundary value problems:
\begin{eqnarray}\label{para-nonlinear}
\begin{array}{rclcl}
\Dev{}{t}u-\nu\Delta u&=&g(x,t,u,\nabla u)&{\rm in}&\Omega\times J,\\
u(x,t)&=&0&{\rm on}&\partial\Omega\times J,\\
u(0)&=&u_0&{\rm in}&\Omega,
\end{array}
\end{eqnarray}
where $g$ is a nonlinear function in $u$ and $u_0$ a function in the space variable $x$ with appropriate assumptions.

As well known, the problems of the form \Ref{para-nonlinear} appear in various kinds of fields in science and technology, and many mathematical and numerical approaches are proposed up to now. 
In this paper, we consider the numerical method to prove the existence of solutions for the problem \Ref{para-nonlinear}. 
Our approach is based on the finite element approximation of a simple heat equation and its constructive a priori error estimates, which is the general principle of our numerical verification method established in \cite{Nakao-NM},\cite{Nakao-JCAM},\cite{Nakao-QNA}, but we presently use a different approximation scheme from these works. 
Namely, we previously used, in \cite{Nakao-NM}, the error estimates given in \cite{Nakao-SIAM} in the finite element Galerkin method with application of an interpolation in time by computing the exact fundamental solution for semidiscretization in space. 
Therefore, we need some complicated numerical algorithms to compute rigorously the matrix exponential. 
This also leads to decrease in computational efficiency and accuracy to realize the desired verification. Here, since we construct a full-discrete approximation scheme by using the tensor product of finite element subspaces for space and time directions, the computational algorithm is much simplified as well as it seems to be very natural and familiar from the numerical point of view. (see \cite{Hashimoto-JJIAM} for details) 
And we also formulated a method to verify the solution on a time evolutional sense, while the results so far had fixed time intervals. The effectiveness of our method is confirmed by some numerical examples for the realistic problems. 
On the other hand, there are some related research works by different approaches on the verified computation of the problem \Ref{para-nonlinear}. For example, in \cite{Miuguchi}, some results are presented based on the semigroup theory and also the problems on the verification of periodic orbits are considered by researchers in dynamical systems, e.g., \cite{Gameiro-3}, using the Fourier spectral method.

\setcounter{equation}{0}
\section{Notations and Preliminaries}
The notations to the spaces and preliminaries in this paper are very similar to that presented in \cite{Hashimoto-JJIAM}, we include these here for the sake of convenience.
\subsection{Notations}
We denote $L^2(\Omega)$ and $H^1(\Omega)$ as the usual Lebesgue and the first order $L^2$-Sobolev spaces on $\Omega$, respectively, and by $\Inner{u}{v}{L^2(\Omega)}:=\int_\Omega u(x)v(x)\,dx$ the natural inner product for $u$, $v\in L^2(\Omega)$. 
By considering the boundary and initial conditions, we define the following subspaces of $H^1_0(\Omega)$ and $V^1(J)$ as
\[ H_0^1(\Omega):=\Set{u \in H^1(\Omega)}{u = 0 \textrm{ on } \partial\Omega} \ \mbox{and} \ V^1(J):=\Set{u \in H^1(J)}{u(0) = 0}, \]
respectively. 
These are Hilbert spaces with inner products
\[ \Inner{u}{v}{H_0^1(\Omega)}:=\Inner{\nabla u}{\nabla v}{L^2(\Omega)^d}\ \mbox{and} \ \Inner{u}{v}{V^1(J)}:=\Inner{\frac{du}{dt}}{\frac{dv}{dt}}{L^2(J)}. \]
Let $X(\Omega)$ be a subspace of $L^2(\Omega)$ defined by $X(\Omega):=\Set{u \in L^2(\Omega)}{\triangle u \in L^2(\Omega)}$. 
We define 
\[
V^1\bigl(J;Y\bigr):=\Set{u \in L^2\bigl(J;Y\bigr)}{\Dev{u}{t} \in L^2\bigl(J;Y\bigr),\ u(0)=0\ {\rm in}\ Y},
\] 
with inner product $\Inner{u}{v}{V^1\bigl(J;Y\bigr)}:=\Inner{\frac{du}{dt}}{\frac{dv}{dt}}{L^2\bigl(J;Y\bigr)}$ for Hilbert space $Y$.
In the following discussion, abbreviations like $L^2H_0^1$ for $L^2\bigl(J;H_0^1(\Omega)\bigr)$ will often be used. 
We set $V(\Omega,J):= V^1\bigl(J;L^2(\Omega)\bigr) \cap L^2\bigl(J;H_0^1(\Omega)\bigr)$.
Moreover, we denote the partial differential operator $\triangle_t:V(\Omega,J) \cap L^2\bigl(J;X(\Omega)\bigr) \to L^2\bigl(J;L^2(\Omega)\bigr)$ by $\triangle_t:=\Dev{}{t}-\nu\triangle$.

Now let $S_h(\Omega)$ be a finite-dimensional subspace of $H_0^1(\Omega)$ dependent on the parameter $h$.
For example, $S_h(\Omega)$ is considered to be a finite element space with mesh size $h$.
Let $n$ be the degree of freedom for $S_h(\Omega)$, and let $\{\phi_i\}_{i=1}^n \subset H_0^1(\Omega)$ be the basis of $S_h(\Omega)$.
Similarly, let $V_k^1(J)$ be an approximation subspace of $V^1(J)$ dependent on the parameter $k$. 
Let $m$ be the degree of freedom for $V_k^1(J)$, and let $\{\psi_i\}_{i=1}^m \subset V_k^1(J)$ be the basis of $V_k^1(J)$.
Let $V^1\bigl(J;S_h(\Omega)\bigr)$ be a subspace of $V(\Omega,J)$ corresponding to the semidiscretized approximation in the spatial direction.
We define the $H_0^1$-projection $P_h^1 u \in S_h(\Omega)$ of any element $u \in H_0^1(\Omega)$ by the following variational equation:
\begin{align}
\Inner{\nabla (u-P_h^1u)}{\nabla v_h}{L^2(\Omega)^d} = 0, \quad \forall v_h \in S_h(\Omega). \label{eq:IPP:Ph1 def}
\end{align}
Similarly, for any element $u \in V^1(J)$, the $V^1$-projection $P^k_1:V^1(J) \to V_k^1(J)$ is defined as follows:
\begin{eqnarray*}
\Inner{\frac{d}{dt}(u-P_1^ku)}{\frac{d}{dt}v_k}{L^2(J)} = 0, \quad \forall v_k \in V_k^1(J).
\end{eqnarray*}

Now let $\Pi_k:V^1(J) \to V_k^1(J)$ be an interpolation operator.
Namely, if the nodal points of $J$ are given by $0=t_0<t_1<\cdots<t_m=T$, then for an arbitrary $u \in V^1(J)$, the interpolation $\Pi_ku$ is defined as the function in $V_k^1(J)$ satisfying:
\begin{align}
u(t_i) = \bigl(\Pi_ku\bigr)(t_i), \quad \forall i \in \{0,\ldots,m\}. \label{eq:IPP:Pik def}
\end{align}
\subsection{Preliminaries}
We know that there exist constants $C_\Omega(h)>0$, $C_J(k)>0$ and $\Cinv(h)>0$ satisfying
\begin{eqnarray*}
\Norm{u-P_h^1u}{H_0^1(\Omega)}
&\leq &C_\Omega(h)\Norm{\triangle u}{L^2(\Omega)}, \quad \forall u\in H_0^1(\Omega) \cap X(\Omega), \\
\Norm{u-\Pi_ku}{L^2(J)}
&\leq& C_J(k)\Norm{u}{V^1(J)}, \quad \forall u\in V^1(J),\\
\Norm{u_h}{H_0^1(\Omega)}
&\leq& C_{inv}(h)\Norm{u_h}{L^2(\Omega)}, \quad \forall u_h\in S_h(\Omega).
\end{eqnarray*}
Moreover, there exists a Poincar\'e constant $C_p>0$ satisfying
\begin{eqnarray*}
\Norm{u}{L^2(\Omega)}
&\leq& C_p\Norm{u}{H^1_0(\Omega)}, \quad \forall u\in H^1_0(\Omega).
\end{eqnarray*}
For example, if $\Omega$ is a bounded rectangular domain in $\mathbb{R}^d$, and $S_h(\Omega)$ is the piecewise bi-linear (Q1) finite element space, then they can be taken by $C_\Omega(h)=\frac{h}{\pi}$ (see, e.g., \cite{Nakao 1998 best}) and $C_{inv}(h)=\frac{\sqrt{12}}{h_{\min}}$, where $h_{\min}$ is the minimum mesh size for $\Omega$ (see, e.g., \cite[Theorem 1.5]{Schultz 1973}).
Moreover, if $V_k^1(J)$ is the P1-finite element space, then it can be taken by $C_J(k)=\frac{k}{\pi}$ (see, e.g., \cite[Theorem 2.4]{Schultz 1973}).

From \cite[Lemma 2.2]{Kinoshita 2011}, if $V_k^1(J)$ is a P1-finite element space (i.e., the basis functions $\psi_i$ are piecewise linear functions), then $P^k_1$ coincides with $\Pi_k$. For any element $u \in V(\Omega,J)$, we define the semidiscrete projection $P_hu \in V^1\bigl(J;S_h(\Omega)\bigr)$ by the following weak form:
\begin{eqnarray*}
\Inner{\Dev{}{t}\bigl(u(t)-P_hu(t)\bigr)}{v_h}{L^2(\Omega)} + \nu\Inner{\nabla \bigl(u(t)-P_hu(t)\bigr)}{\nabla v_h}{L^2(\Omega)^d} = 0 \label{eq:IPP:Ph def} \\ \quad \forall v_h \in S_h(\Omega)~~\AEW t \in J, \nonumber 
\end{eqnarray*}
where $\AEW$means an abbreviation for `almost everywhere'.

Let define the space $S_h^k(\Omega,J)$ as the tensor product $V_k^1(J) \otimes S_h(\Omega)$ which corresponds to a full discretization subspace of $V(\Omega,J)$, and let $\{\varphi_i\}_{ i=1}^{mn}$ be a basis of $S_h^k(\Omega,J)$.
Moreover, we define the full-discretization operator $P_h^k:V(\Omega,J) \to S_h^k(\Omega,J)$ by $P_h^k:=\Pi_kP_h$.
In addition, we denote the matrix norm induced from the Euclidean 2-norm in $\mathbb{R}^{mn}$ by $\|\cdot\|_{E}$, and the transposed matrix of the matrix ${\bf X}$ by ${\bf X}^{\rm T}$.

First, we present some known results on the a priori error estimates for the full-discretization operator $P_h^k$.
\begin{theorem}(see Theorem 5.5, 5.6, and proof of Theorem 4.6 in \cite{Nakao-SIAM})\label{error-estimate-base}
For an arbitrary $u\in V(\Omega,J) \cap L^2\bigl(J;X(\Omega)\bigr)$, we have the following estimations
\begin{eqnarray*}
\Norm{u-P_h^ku}{L^2\bigl(J;H_0^1(\Omega)\bigr)}
&\leq& C_1(h,k)\Norm{\triangle_t u}{L^2\bigl(J;L^2(\Omega)\bigr)},\label{C1hk}\\
\Norm{u-P_h^ku}{L^2\bigl(J;L^2(\Omega)\bigr)}
&\leq& C_0(h,k)\Norm{\triangle_t u}{L^2\bigl(J;L^2(\Omega)\bigr)},\label{C0hk}\\
\Norm{u(T)-P_hu(T)}{L^2(\Omega)}
&\leq& c_0(h)\Norm{\triangle_t u}{L^2\bigl(J;L^2(\Omega)\bigr)},\label{c0hk}
\end{eqnarray*}
where $C_1(h,k):=\frac{2}{\nu}C_\Omega(h) + \Cinv(h)C_J(k)$, $C_0(h,k):=\frac{8}{\nu}C_\Omega(h)^2+C_J(k)$ and $c_0(h):=\sqrt{\frac{8}{\nu}}C_\Omega(h)$.
\end{theorem}

Next, we define the bilinear form $a_0(\cdot,\cdot)$ by
\begin{eqnarray*}\label{bilinear-form}
a_0(\phi,\psi)&:=&
\Inner{\Dev{}{t}\phi}{\Dev{}{t}\psi}{L^2\bigl(J;L^2(\Omega)\bigr)}
+\nu\Inner{\nabla\phi}{\Dev{}{t}\nabla\psi}{L^2\bigl(J;L^2(\Omega)\bigr)^d}
\end{eqnarray*}
for $\phi\in V(\Omega,J)$ and $\psi\in V^1\bigl(J;H^1_0(\Omega)\bigr)$.
Then, for any element $u \in V(\Omega,J)$, we define the full-discrete projection $Q_h^ku \in S_h^k(\Omega,J)$ by the following weak form:
\begin{eqnarray*}\label{def-projection}
a_0(u-Q_h^ku,v_h^k) = 0, \quad \forall v_h^k \in S_h^k(\Omega,J).
\end{eqnarray*}
It is readily seen that, if $u\in V(\Omega,J) \cap L^2\bigl(J;X(\Omega)\bigr)$ then
we have
\begin{eqnarray}\label{FEM-solution}
a_0(Q_h^ku,v_h^k)&=&\Inner{\triangle_t u}{\Dev{}{t}v_h^k}{L^2\bigl(J;L^2(\Omega)\bigr)}\quad \forall v_h^k \in S_h^k(\Omega,J).
\end{eqnarray} 

In order to present the error estimates for $Q_h^ku$, we need to define 
several kinds of matrices as below.

The matrices ${\bf A}$ and ${\bf M}$ in $\mathbb{R}^{mn \times mn}$ are defined by 
$$
{\bf A}_{i,j}:=\Inner{\Dev{}{t}\varphi_j}{\Dev{}{t}\varphi_i}{L^2\bigl(J;L^2(\Omega)\bigr)},\quad 
{\bf M}_{i,j}:=\Inner{\nabla\varphi_j}{\nabla\varphi_i}{L^2\bigl(J;L^2(\Omega)\bigr)^d},\quad
\forall i,j \in \{1,\ldots,mn\},
$$
respectively.
Since matrices ${\bf A}$ and ${\bf M}$ are symmetric and positive definite, we can denote them by using the Cholesky decomposition as ${\bf A}={\bf A}^{\frac{1}{2}}{\bf A}^{\frac{\rm T}{2}}$ and ${\bf M}={\bf M}^{\frac{1}{2}}{\bf M}^{\frac{\rm T}{2}}$, respectively.
Also, we define ${\bf B}$ in $\mathbb{R}^{mn \times mn}$ by
$$
{\bf B}_{i,j}:=\Inner{\nabla\varphi_j}{\Dev{}{t}\nabla\varphi_i}{L^2\bigl(J;L^2(\Omega)\bigr)^d} \quad\forall i,j \in \{1,\ldots,mn\}.
$$
Further define ${\bf W}$ and ${\bf U}$ in $\mathbb{R}^{mn \times mn}$ by 
$$
{\bf U}_{i,j}:=\Inner{\varphi_j}{\varphi_i}{L^2\bigl(J;L^2(\Omega)\bigr)},\quad
{\bf W}_{i,j}:=\Inner{\Dev{}{t}\nabla\varphi_j}{\Dev{}{t}\nabla\varphi_i}{L^2\bigl(J;L^2(\Omega)\bigr)^d} \quad 
\forall i,j \in \{1,\ldots,mn\},
$$
respectively.
Note that, since matrices ${\bf U}$ and ${\bf W}$ are symmetric and positive definite, similarly as above, we have the expressions such as ${\bf U}={\bf U}^{\frac{1}{2}}{\bf U}^{\frac{T}{2}}$ and ${\bf W}={\bf W}^{\frac{1}{2}}{\bf W}^{\frac{T}{2}}$, respectively.
Moreover, we define the symmetric and positive semidefinite matrix ${\bf Y}={\bf Y}^{\frac{1}{2}}{\bf Y}^{\frac{T}{2}}$ in $\mathbb{R}^{mn \times mn}$ by
$$
{\bf Y}_{i,j}:=\Inner{\varphi_j(\cdot,T)}{\varphi_i(\cdot,T)}{L^2(\Omega)}\quad\forall i,j \in \{1,\ldots,mn\}.
$$

Now let
\begin{eqnarray*}
\gamma_1&:=&\nu\|{\bf M}^{\frac{T}{2}}({\bf A}+\nu {\bf B})^{-1}{\bf W}^{\frac{1}{2}}\|_E,\\
\gamma_0&:=&\nu\|{\bf U}^{\frac{T}{2}}({\bf A}+\nu {\bf B})^{-1}{\bf W}^{\frac{1}{2}}\|_E,\\
\gamma_T&:=&\nu\|{\bf Y}^{\frac{T}{2}}({\bf A}+\nu {\bf B})^{-1}{\bf W}^{\frac{1}{2}}\|_E.
\end{eqnarray*}
Then we have the following result.
\begin{theorem}(Theorem 6 in \cite{Hashimoto-JJIAM})\label{thm:error}
Assume that $V_k^1(J)$ is the P1 finite element space.
For an arbitrary $u\in V(\Omega,J) \cap L^2\bigl(J;X(\Omega)\bigr)$, we have the following estimations.
\begin{eqnarray*}
\Norm{u-Q_h^ku}{L^2\bigl(J;H_0^1(\Omega)\bigr)}
&\leq& \tilde{C}_1(h,k)\Norm{\triangle_t u}{L^2\bigl(J;L^2(\Omega)\bigr)},\\
\Norm{u-Q_h^ku}{L^2\bigl(J;L^2(\Omega)\bigr)}
&\leq& \tilde{C}_0(h,k)\Norm{\triangle_t u}{L^2\bigl(J;L^2(\Omega)\bigr)},\\
\Norm{u(T)-Q_h^ku(T)}{L^2(\Omega)}
&\leq& \tilde{c}_0(h)\Norm{\triangle_t u}{L^2\bigl(J;L^2(\Omega)\bigr)},
\end{eqnarray*}
where 
\begin{eqnarray*}
\tilde{C}_1(h,k)&:=&C_1(h,k)+C_J(k)C_{inv}(h)\gamma_1,\\
\tilde{C}_0(h,k)&:=&C_0(h,k)+C_J(k)C_{inv}(h)\gamma_0,\\
\tilde{c}_0(h,k)&:=&c_0(h)+C_J(k)C_{inv}(h)\gamma_T.
\end{eqnarray*}
\end{theorem}
Note that these constants $\gamma_1,\gamma_0$ and $\gamma_T$ can be rigorously estimated by using appropriate numerical computations with self-validating algorithms, e.g., a tool box in MATLAB developed by Rump \cite{Rump INTLAB} and so on.

\setcounter{equation}{0}
\section{Norm estimation of the linearized operators}
In order to verify a solution of the problem \Ref{para-nonlinear} by Newton's method, we first consider the following linear parabolic problem with homogeneous initial condition.
\begin{eqnarray}\label{para-linear}
\begin{array}{rclcl}
\Dev{}{t}w-\nu\Delta w+b\cdot\nabla w+cw&=&f&{\rm in}&\Omega\times J,\\
w(x,t)&=&0&{\rm on}&\partial\Omega\times J,\\
w(0)&=&0&{\rm in}&\Omega,
\end{array}
\end{eqnarray}
where $f\in L^2\bigl(J;L^2(\Omega)\bigr)$ is a given function, and we assume that $b\in L^\infty\bigl(J;L^\infty(\Omega)\bigr)^d$, $c\in L^\infty\bigl(J;L^\infty(\Omega)\bigr)$.

Let $\Delta_t^{-1}:L^2\bigl(J;L^2(\Omega)\bigr)\rightarrow V(\Omega,J)\cap L^2\bigl(J;X(\Omega)\bigr)$ be a solution operator such that
$$
\psi=\Delta_t^{-1}\phi
\quad\Leftrightarrow\quad
\begin{array}{rclcl}
\Dev{}{t}\psi-\nu\Delta \psi&=&\phi&{\rm in}&\Omega\times J,\\
\psi(x,t)&=&0&{\rm on}&\partial\Omega\times J,\\
\psi(0)&=&0&{\rm in}&\Omega,
\end{array}
$$
where $\phi\in L^2\bigl(J;L^2(\Omega)\bigr)$.
Then the problem \Ref{para-linear} can be written as the following fixed point form 
\begin{eqnarray}\label{fixed point}
w&=&\Delta_t^{-1}(-b\cdot\nabla w-cw+f).
\end{eqnarray}
Moreover, the problem \Ref{fixed point} is decomposed as
\begin{eqnarray}
Q_h^kw&=&Q_h^k\Delta_t^{-1}(-b\cdot\nabla w-cw+f),\label{fix:projection}\\
(I-Q_h^k)w&=&(I-Q_h^k)\Delta_t^{-1}(-b\cdot\nabla w-cw+f).\nonumber
\end{eqnarray}

We now define the bilinear form ${a}(\cdot,\cdot)$ by
\begin{eqnarray*}
{a}(\phi,\psi)&:=&
{a}_0(\phi,\psi)+\Inner{b\cdot\nabla\phi+c\phi}{\Dev{}{t}\psi}{L^2\bigl(J;L^2(\Omega)\bigr)}
\quad\mbox{for}\quad\phi\in V(\Omega,J),\ \psi\in V^1\bigl(J;H^1_0(\Omega)\bigr).
\end{eqnarray*}
Then by \Ref{FEM-solution}, the equation \Ref{fix:projection} is equivalent to the following.
\begin{eqnarray*}
{a}(Q_h^kw,v_h^k)=\Inner{-b\cdot\nabla w_\bot-c w_\bot+f}{\Dev{}{t}v_h^k}{L^2\bigl(J;L^2(\Omega)\bigr)} \quad \forall v_h^k\in S_h^k(\Omega,J),
\end{eqnarray*}
where $w_\bot:=w-Q_h^k w$.
We now define $r_h^k\in S_h^k(\Omega,J)$ satisfying
$$
\Inner{\Dev{}{t}r_h^k}{\Dev{}{t}v_h^k}{L^2\bigl(J;L^2(\Omega)\bigr)}
=\Inner{-b\cdot\nabla w_\bot-c w_\bot+f}{\Dev{}{t}v_h^k}{L^2\bigl(J;L^2(\Omega)\bigr)} \quad \forall v_h^k\in S_h^k(\Omega,J).
$$
Note that from the definition of $r_h^k$, it follows that
\begin{eqnarray}\label{rhk}
\Norm{r_h^k}{V^1\bigl(J;L^2(\Omega)\bigr)}
\le\Norm{-b\cdot\nabla w_\bot -c w_\bot+f}{L^2\bigl(J;L^2(\Omega)\bigr)}.
\end{eqnarray}
Also \Ref{fix:projection} is written as 
\begin{eqnarray}\label{linearized-finite}
{a}(Q_h^kw,v_h^k)=\Inner{\Dev{}{t}r_h^k}{\Dev{}{t}v_h^k}{L^2\bigl(J;L^2(\Omega)\bigr)}\quad\forall v_h^k\in S_h^k(\Omega,J).
\end{eqnarray}
We define the matrix ${\bf G} \in \mathbb{R}^{mn \times mn}$ by
${\bf G}_{i,j}:={a}(\varphi_j,\varphi_i)$ for all $i,j \in \{1,\ldots,mn\}$.
From the fact that $Q_h^kw$ and $r_h^k$ belong to $S_h^k(\Omega,J)$, there exist coefficient vectors $\mathfrak{w}:=(\mathfrak{w}_1,\ldots,\mathfrak{w}_{mn})^T$ and $\mathfrak{r}:=(\mathfrak{r}_1,\ldots,\mathfrak{r}_{mn})^T$ in $\mathbb{R}^{mn}$ such that
$Q_h^kw= \sum_{i=1}^{mn} \mathfrak{w}_i\varphi_i = \varphi^T\mathfrak{w}$ and $r_h^k= \sum_{i=1}^{mn} \mathfrak{r}_i\varphi_i = \varphi^T\mathfrak{r}$.
Then, the variational equation \Ref{linearized-finite} can be rewritten as the following matrix from
\begin{eqnarray}\label{GMat-e}
{\bf G}\mathfrak{w}={\bf A}\mathfrak{r}.
\end{eqnarray}
From \Ref{GMat-e}, we can obtain the following estimations.
\begin{eqnarray*}
\Norm{Q_h^kw}{L^2\bigl(J;H^1_0(\Omega)\bigr)}
&=&\|{\bf M}^{\frac{T}{2}}\mathfrak{w}\|_E
=\|{\bf M}^{\frac{T}{2}}{\bf G}^{-1}{\bf A}\mathfrak{r}\|_E
\le\|{\bf M}^{\frac{T}{2}}{\bf G}^{-1}{\bf A}^{\frac{1}{2}}\|_E\|{\bf A}^{\frac{T}{2}}\mathfrak{r}\|_E,\\
\Norm{Q_h^kw}{L^2\bigl(J;L^2(\Omega)\bigr)}
&=&\|{\bf U}^{\frac{T}{2}}\mathfrak{w}\|_E
=\|{\bf U}^{\frac{T}{2}}{\bf G}^{-1}{\bf A}\mathfrak{r}\|_E
\le\|{\bf U}^{\frac{T}{2}}{\bf G}^{-1}{\bf A}^{\frac{1}{2}}\|_E\|{\bf A}^{\frac{T}{2}}\mathfrak{r}\|_E,\\
\Norm{Q_h^kw(T)}{L^2(\Omega)}
&=&\|{\bf Y}^{\frac{T}{2}}\mathfrak{w}\|_E
=\|{\bf Y}^{\frac{T}{2}}{\bf G}^{-1}{\bf A}\mathfrak{r}\|_E
\le\|{\bf Y}^{\frac{T}{2}}{\bf G}^{-1}{\bf A}^{\frac{1}{2}}\|_E\|{\bf A}^{\frac{T}{2}}\mathfrak{r}\|_E.
\end{eqnarray*}
Moreover, using the results in Theorem \ref{thm:error}, we have by \Ref{rhk}
\begin{eqnarray*}
\|{\bf A}^{\frac{T}{2}}\mathfrak{r}\|_E
&=&\Norm{r_h^k}{V^1\bigl(J;L^2(\Omega)\bigr)}\\
&\le&\Norm{-b\cdot\nabla w_\bot-c w_\bot+f}{L^2\bigl(J;L^2(\Omega)\bigr)}\\
&\le&\left(\tilde{C}_1(h,k)C_b+\tilde{C}_0(h,k)C_c\right)\Norm{\Delta_tw}{L^2\bigl(J;L^2(\Omega)\bigr)}+\Norm{f}{L^2\bigl(J;L^2(\Omega)\bigr)},
\end{eqnarray*}
where $C_b:=\sqrt{\sum_{i=1}^d\Norm{b_i}{L^\infty\bigl(J;L^\infty(\Omega)\bigr)}^2}$ and $C_c:=\Norm{c}{L^\infty\bigl(J;L^\infty(\Omega)\bigr)}$. 
Thus letting 
\begin{eqnarray*}
M_1:=\|{\bf M}^{\frac{T}{2}}{\bf G}^{-1}{\bf A}^{\frac{1}{2}}\|_E,\quad
M_0:=\|{\bf U}^{\frac{T}{2}}{\bf G}^{-1}{\bf A}^{\frac{1}{2}}\|_E,\quad
M_T:=\|{\bf Y}^{\frac{T}{2}}{\bf G}^{-1}{\bf A}^{\frac{1}{2}}\|_E,
\end{eqnarray*}
and $\tau(h,k):=\tilde{C}_1(h,k)C_b+\tilde{C}_0(h,k)C_c$, it follows that
\begin{eqnarray}
\Norm{Q_h^kw}{L^2\bigl(J;H^1_0(\Omega)\bigr)}
&\le&M_1\left(\tau(h,k)\Norm{\Delta_t w}{L^2\bigl(J;L^2(\Omega)\bigr)}+\Norm{f}{L^2\bigl(J;L^2(\Omega)\bigr)}\right),\label{est1}\\
\Norm{Q_h^kw}{L^2\bigl(J;L^2(\Omega)\bigr)}
&\le&M_0\left(\tau(h,k)\Norm{\Delta_t w}{L^2\bigl(J;L^2(\Omega)\bigr)}+\Norm{f}{L^2\bigl(J;L^2(\Omega)\bigr)}\right),\label{est2}\\
\Norm{Q_h^kw(T)}{L^2(\Omega)}
&\le&M_T\left(\tau(h,k)\Norm{\Delta_t w}{L^2\bigl(J;L^2(\Omega)\bigr)}+\Norm{f}{L^2\bigl(J;L^2(\Omega)\bigr)}\right).\label{est3}
\end{eqnarray}
On the other hand, using \Ref{est1}, \Ref{est2} and Theorem \ref{thm:error}, we obtain
\begin{eqnarray}
\Norm{\Delta_t w}{L^2\bigl(J;L^2(\Omega)\bigr)}
&=&\Norm{-b\cdot\nabla w -cw+f}{L^2\bigl(J;L^2(\Omega)\bigr)}\nonumber\\
&\le&C_b\left(\Norm{Q_h^k w}{L^2\bigl(J;H^1_0(\Omega)\bigr)}+\Norm{w_\bot}{L^2\bigl(J;H^1_0(\Omega)\bigr)}\right)\nonumber\\
&&+C_c\left(\Norm{Q_h^k w}{L^2\bigl(J;L^2(\Omega)\bigr)}+\Norm{w_\bot}{L^2\bigl(J;L^2(\Omega)\bigr)}\right)+\Norm{f}{L^2\bigl(J;L^2(\Omega)\bigr)}\nonumber\\
&\le&\tau(h,k){\cal E}\Norm{\Delta_t w}{L^2\bigl(J;L^2(\Omega)\bigr)}+{\cal E}\Norm{f}{L^2\bigl(J;L^2(\Omega)\bigr)},\label{est4}
\end{eqnarray}
where ${\cal E}:=M_1C_b+M_0C_c+1$.

Hence, we have the following result.
\begin{lemma}\label{le:Cdelta}
Let $w\in V(\Omega,J) \cap L^2\bigl(J;X(\Omega)\bigr)$ be a solution of the problem \Ref{para-linear}.
Assuming that $\kappa(h,k)\equiv\tau(h,k){\cal E}<1$ we have 
\begin{eqnarray*}
\Norm{\Delta_t w}{L^2\bigl(J;L^2(\Omega)\bigr)}
&\le&C_{\Delta}\Norm{f}{L^2\bigl(J;L^2(\Omega)\bigr)},\quad\mbox{where}\quad C_{\Delta}:=\frac{{\cal E}}{1-\kappa(h,k)}.
\end{eqnarray*}
Moreover, it follows that
\begin{eqnarray*}
\Norm{Q_h^kw}{L^2\bigl(J;H^1_0(\Omega)\bigr)}
&\le&M_1C_{Q_h^k}\Norm{f}{L^2\bigl(J;L^2(\Omega)\bigr)},\\
\Norm{Q_h^kw}{L^2\bigl(J;L^2(\Omega)\bigr)}
&\le&M_0C_{Q_h^k}\Norm{f}{L^2\bigl(J;L^2(\Omega)\bigr)},\\
\Norm{Q_h^kw(T)}{L^2(\Omega)}
&\le&M_TC_{Q_h^k}\Norm{f}{L^2\bigl(J;L^2(\Omega)\bigr)},
\end{eqnarray*}
where $C_{Q_h^k}:=\tau(h,k)C_{\Delta}+1$.
\end{lemma}
\begin{proof}
From \Ref{est1}, \Ref{est2}, \Ref{est3} and \Ref{est4}, this proof is completed.
\end{proof}

Using above lemma, we can obtain the following result.
\begin{lemma}\label{le:Cdt}
Under the same assumption in Lemma \ref{le:Cdelta}, we have the following estimations:
\begin{eqnarray*}
\Norm{w}{V^1\bigl(J;L^2(\Omega)\bigr)}
\le C_{\Delta}\Norm{f}{L^2\bigl(J;L^2(\Omega)\bigr)},\quad
\Norm{w(t)}{H^1_0(\Omega)}
\le\sqrt{\frac{1}{\nu}}C_{\Delta}\Norm{f}{L^2\bigl(J;L^2(\Omega)\bigr)},\quad\forall t\in J.
\end{eqnarray*}
\end{lemma}
\begin{proof}
Since $w(x,0)=0$ in $\Omega$, observe that for $t\in J$
\begin{eqnarray*}
\int_0^t\Norm{\Dev{}{t}w(s)}{L^2(\Omega)}^2 ds+\frac{\nu}{2}\Norm{w(t)}{H^1_0(\Omega)}^2
&=&\int_0^t\Inner{\Delta_tw(s)}{\Dev{}{t}w(s)}{L^2(\Omega)} ds.\\
&\le&\frac{1}{2}\int_0^t\Norm{\Delta_tw(s)}{L^2(\Omega)}^2 ds+\frac{1}{2}\int_0^t\Norm{\Dev{}{t}w(s)}{L^2(\Omega)}^2 ds\\
&\le&\frac{1}{2}\Norm{\Delta_t w}{L^2\bigl(J;L^2(\Omega)\bigr)}^2+\frac{1}{2}\int_0^t\Norm{\Dev{}{t}w(s)}{L^2(\Omega)}^2 ds,
\end{eqnarray*}
which implies
$$
\int_0^t\Norm{\Dev{}{t}w(s)}{L^2(\Omega)}^2 ds+\nu\Norm{w(t)}{H^1_0(\Omega)}^2
\le\Norm{\Delta_t w}{L^2\bigl(J;L^2(\Omega)\bigr)}^2.
$$
Therefore, this proof is completed.
\end{proof}

Finally we get the following main results in this section.
\begin{theorem}\label{thm:main}
Under the same assumption in Lemma \ref{le:Cdelta}, we have the following estimations:
\begin{eqnarray*}
\Norm{w}{L^2\bigl(J;H^1_0(\Omega)\bigr)}
&\le&{\cal M}_1\Norm{f}{L^2\bigl(J;L^2(\Omega)\bigr)},\quad
{\cal M}_1:=M_1C_{Q_h^k}+\tilde{C}_1(h,k)C_{\Delta},\\
\Norm{w}{L^2\bigl(J;L^2(\Omega)\bigr)}
&\le&{\cal M}_0\Norm{f}{L^2\bigl(J;L^2(\Omega)\bigr)},\quad
{\cal M}_0:=M_0C_{Q_h^k}+\tilde{C}_0(h,k)C_{\Delta},\\
\Norm{w(T)}{L^2(\Omega)}
&\le&{\cal M}_T\Norm{f}{L^2\bigl(J;L^2(\Omega)\bigr)},\quad
{\cal M}_T:=M_TC_{Q_h^k}+\tilde{c}_0(h,k)C_{\Delta}.
\end{eqnarray*}
\end{theorem}
\begin{proof}
By using the following triangle inequalities,
\begin{eqnarray*}
\Norm{w}{L^2\bigl(J;H^1_0(\Omega)\bigr)}
&\le&\Norm{Q_h^kw}{L^2\bigl(J;H^1_0(\Omega)\bigr)}+\Norm{w-Q_h^kw}{L^2\bigl(J;H^1_0(\Omega)\bigr)}\\
\Norm{w}{L^2\bigl(J;L^2(\Omega)\bigr)}
&\le&\Norm{Q_h^kw}{L^2\bigl(J;L^2(\Omega)\bigr)}+\Norm{w-Q_h^kw}{L^2\bigl(J;L^2(\Omega)\bigr)}\\
\Norm{w(T)}{L^2(\Omega)}
&\le&\Norm{Q_h^kw(T)}{L^2(\Omega)}+\Norm{w(T)-Q_h^kw(T)}{L^2(\Omega)},
\end{eqnarray*}
we obtain the desired result from Theorem \ref{thm:error} and Lemma \ref{le:Cdelta}.
\end{proof}
\setcounter{equation}{0}
\section{A numerical verification algorithm for nonlinear problems}
In this section, we present a numerical verification method of solutions for nonlinear problems \Ref{para-nonlinear} with $g(u):=g(x,t,u,\nabla u)$. Dividing interval $J$ into $m$ subintervals, we use a time evolving algorithm in the below. 

Now, in order to get an appropriate approximate solution for the problem \Ref{para-nonlinear} on $\Omega\times J$, we use a finite element subspace $\bar{S}_h^k(\Omega,J):=\bar{V}_k^1(J) \otimes \bar{S}_h(\Omega)$, where we suppose that $\bar{V}_k^1(J)\subset H^1(J)$ and $\bar{S}_h(\Omega)\subset H^2(\Omega)$, respectively.
Let $\bar{u}_{h,i}^{k}\in \bar{S}_h^k(\Omega,J_i)$ be an approximate solution of the problem \Ref{para-nonlinear} on $\Omega\times J_i$, where $J_i=(t_{i-1},t_{i})\subset\mathbb{R}$ is a subinterval of $J$ with $t_0 =0$, and $T_i\equiv |J_i|=t_i-t_{i-1}$ for $i=1,\cdots,l$.

First we consider the problem \Ref{para-nonlinear} on $\Omega\times J_i$.
Letting $\bar{u}:=u-\bar{u}_{h,i}^{k}$, the problem \Ref{para-nonlinear} is equivalent to the following residual equation
\begin{eqnarray}\label{para-res}
\begin{array}{rclcl}
\frac{\partial}{\partial t}\bar{u}-\nu\Delta \bar{u}&=&g(\bar{u}+\bar{u}_{h,i}^{k})-g(\bar{u}_{h,i}^{k})+\delta_i&{\rm in}&\Omega\times J_i,\\
\bar{u}(x,t)&=&0&{\rm on}&\partial\Omega\times J_i,\\
\bar{u}(t_{i-1})&=&\epsilon_i&{\rm in}&\Omega,
\end{array}
\end{eqnarray}
where $\epsilon_1=u_{0}-\bar{u}_{h,1}^k(t_0)$ and $\delta_i:=g(\bar{u}_{h,i}^{k})-\frac{\partial}{\partial t}\bar{u}_{h,i}^{k}+\nu\Delta\bar{u}_{h,i}^{k}$ is a residual function.

We now define the operators $\mathcal{L}_{i}: H^1(J_i, L^2(\Omega)) \cap L^2\bigl(J_i;X(\Omega)\bigr) \to L^2\bigl(J_i;L^2(\Omega)\bigr)$ by
\begin{eqnarray*}
\mathcal{L}_{i}&\equiv&\frac{\partial}{\partial t}-\nu\Delta-g^{\prime}[\bar{u}_{h,i}^{k}],\quad i=1,\cdots,l,
\end{eqnarray*}
where $g^{\prime}[\bar{u}_{h,i}^{k}]$ denote the Fr\'{e}chet derivative of $g$ at $\bar{u}_{h,i}^{k}$. 
In general, the linearlized operator of the nonlinear problem \Ref{para-nonlinear} can be represented as the left-hand side of the first equation in \Ref{para-linear}. Namely, we may denote
\begin{eqnarray*}\label{linearized-full}
\mathcal{L}_{i} = \Dev{}{t}-\nu\Delta+b_i\cdot\nabla+c_i,
\end{eqnarray*}
where coefficient functions $b_i$ and $c_i$ imply the restriction of the corresponding $b$ and $c$ in \Ref{para-linear} to the domain $\Omega\times J_i$.

\subsection{A Newton-type formulation}
Using the operators $\mathcal{L}_{i}$, a solution $\bar{u}$ of the problem \Ref{para-res} can be decomposed as $\bar{u}=v+w$ by using solutions $v$ and $w$ of
\begin{eqnarray}\label{para-res-v-newton}
\begin{array}{rclcl}
\mathcal{L}_{i}v&=&0&{\rm in}&\Omega\times J_i,\\
v(x,t)&=&0&{\rm on}&\partial\Omega\times J_i,\\
v(t_{i-1})&=&\epsilon_i&{\rm in}&\Omega,
\end{array}
\end{eqnarray}
and
\begin{eqnarray}\label{para-res-w-newton}
\begin{array}{rclcl}
\mathcal{L}_{i}w&=&g_i(w)&{\rm in}&\Omega\times J_i,\\
w(x,t)&=&0&{\rm on}&\partial\Omega\times J_i,\\
w(t_{i-1})&=&0&{\rm in}&\Omega,
\end{array}
\end{eqnarray}
respectively, where $g_i(w)\equiv g(v+w+\bar{u}_{h,i}^{k})-g(\bar{u}_{h,i}^{k})-g^{\prime}[\bar{u}_{h,i}^{k}](v+w)+\delta_i$ for $i=1,\cdots,l$.
Note that the solution $v$ of \Ref{para-res-v-newton} can be determined independently of $w$ in \Ref{para-res-w-newton}. Therefore, if the solution $v$ of the linear equation \Ref{para-res-v-newton} is numerically verified, then the problem \Ref{para-res} can be reduced to find a solution $w$ to the nonlinear problem \Ref{para-res-w-newton}. 
And the problem \Ref{para-res-w-newton} is rewritten as the following fixed-point equation of the compact map $\mathcal{L}_{i}^{-1}$:
\begin{eqnarray}\label{fixed-point}
w=\mathcal{L}_{i}^{-1}g_i(w).
\end{eqnarray}
Here, the map $\mathcal{L}_{i}^{-1}: L^2\bigl(J_i;L^2(\Omega)\bigr) \to V(\Omega,J_i) \cap L^2\bigl(J_i;X(\Omega)\bigr)$ is considered as the solution operator for the linear parabolic problem with homogeneous initial condition corresponding to the problem \Ref{para-linear}. 
For any positive constants $\alpha_i$ and $\beta_i$, we define the candidate set $W_{\alpha_i,\beta_i}$ as
$$
W_{\alpha_i,\beta_i}(\Omega,J_i):=\left\{w\in V(\Omega,J_i)\ ;\ \Norm{w}{L^2\bigl(J_i;H^1_0(\Omega)\bigr)}\le\alpha_i,\ \Norm{w}{V^1\bigl(J_i;L^2(\Omega)\bigr)}\le\beta_i\right\}.
$$
Taking notice of the continuity of the map $\mathcal{L}_{i}^{-1}$ on the space $L^2(J_i;H^1_0(\Omega))$, from the Schauder fixed-point theorem, if the set $W_{\alpha_i,\beta_i}(\Omega,J_i)$ satisfies
\begin{eqnarray}\label{include}
\mathcal{L}_{i}^{-1}g_i(W_{\alpha_i,\beta_i}(\Omega,J_i))\subset W_{\alpha_i,\beta_i}(\Omega,J_i),
\end{eqnarray}
then a fixed-point of \Ref{fixed-point} exists in the set $\overline{W_{\alpha_i,\beta_i}(\Omega,J_i)}$, where $\overline{W_{\alpha_i,\beta_i}(\Omega,J_i)}$ stands for the closure of the set $W_{\alpha_i,\beta_i}(\Omega,J_i)$ in $L^2(J_i;H^1_0(\Omega))$. 
Moreover, for an arbitrary $w \in W_{\alpha_i,\beta_i}(\Omega,J_i)$, setting $\tilde{w} : = \mathcal{L}_{i}^{-1}g_i(w)$ by Lemma \ref{le:Cdt} and Theorem \ref{thm:main} it holds that
\begin{eqnarray}\label{verification-condition}
\begin{array}{rclcl}
\Norm{\tilde{w}}{L^2\bigl(J_i;H^1_0(\Omega)\bigr)}
&\le&{\cal M}_{1_i}\Norm{g_i(w)}{L^2\bigl(J_i;L^2(\Omega)\bigr)},\\
\Norm{\tilde{w}}{L^2\bigl(J_i;L^2(\Omega)\bigr)}
&\le&{\cal M}_{0_i}\Norm{g_i(w)}{L^2\bigl(J_i;L^2(\Omega)\bigr)},\\
\Norm{\tilde{w}(t_i)}{L^2(\Omega)}
&\le&{\cal M}_{t_i}\Norm{g_i(w)}{L^2\bigl(J_i;L^2(\Omega)\bigr)},\\
\Norm{\tilde{w}(t)}{H^1_0(\Omega)}
&\le&\sqrt{\frac{1}{\nu}}C_{\Delta_i}\Norm{g_i(w)}{L^2\bigl(J_i;L^2(\Omega)\bigr)},\quad\forall t\in J_i,\\
\Norm{\tilde{w}}{V^1\bigl(J_i;L^2(\Omega)\bigr)}
&\le&C_{\Delta_i}\Norm{g_i(w)}{L^2\bigl(J_i;L^2(\Omega)\bigr)}.
\end{array}
\end{eqnarray}
Therefore, defining the function $G(\alpha_i,\beta_i)$ of $\alpha_i$ and $\beta_i$ satisfying 
\begin{eqnarray}\label{G}
\sup_{w\in W_{\alpha_i,\beta_i}(\Omega_1,J_i)}\Norm{g_i(w)}{L^2\bigl(J_i;L^2(\Omega_1)\bigr)}\le G(\alpha_i,\beta_i),
\end{eqnarray}
the sufficient condition of \Ref{include} is given as follows:
\begin{eqnarray}\label{condition}
{\cal M}_{1_i}G(\alpha_i,\beta_i)<\alpha_i,\quad
C_{\Delta_i}G(\alpha_i,\beta_i)<\beta_i.
\end{eqnarray}
\subsection{The estimate of $v$ for the linear problem}

Let $v$ be a solution of the problem \Ref{para-res-v-newton}, and let $\hat{v}$ be a solution of the following problem:
\begin{eqnarray}\label{para-res-v}
\begin{array}{rclcl}
\Dev{}{t}\hat{v}-\nu\Delta\hat{v}&=&0&{\rm in}&\Omega\times J_i,\\
\hat{v}(x,t)&=&0&{\rm on}&\partial\Omega\times J_i,\\
\hat{v}(t_{i-1})&=&\epsilon_i&{\rm in}&\Omega.
\end{array}
\end{eqnarray}

Furthermore, by using the above $\hat{v}$, we define the function $v_0$ as a solution of the following linear equation with homogeneous initial condition: 
\begin{eqnarray}\label{para-res-v*-newton}
\begin{array}{rclcl}
\mathcal{L}_{i}v_0&=&-b_{i}\cdot\nabla\hat{v}-c_{i}\hat{v}&{\rm in}&\Omega\times J_i,\\
v_0(x,t)&=&0&{\rm on}&\partial\Omega\times J_i,\\
v_0(t_{i-1})&=&0&{\rm in}&\Omega.
\end{array}
\end{eqnarray}
Then the solution $v$ of \Ref{para-res-v-newton} is written as $v=\hat{v}+v_0$.

Now, noting that, by using the well known spectral theory, e.g., \cite{Barbu 1998}, \cite{Miuguchi} etc., the solution $\hat{v}$ of \Ref{para-res-v} is represented as follows:
$$
\hat{v}(t) = \exp(-(t-t_{i-1})A)\epsilon_i \quad t \in J_i,
$$
where $\exp(-tA)_{t\ge0}$ means a semigroup generated by $A\equiv-\nu\Delta$.

Let $\lambda_{\min} >0$ be the smallest eigenvalue of $A$, for example, if $\Omega=(0,1)$ then $\lambda_{\min}=\nu\pi^2$. ($\lambda_{\min}=2\nu\pi^2$ if $\Omega=(0,1)^2$). 
Here, we denote $\rho(T_i): =\exp(-\lambda_{\min}T_i)$ and $\rho_{\Omega}(T_i):=\sqrt{\frac{1}{2\lambda_{\min}}(1-\rho(2T_i))}$. \\

We set constants $C_{b_i}:=\sqrt{\sum_{j=1}^d\Norm{b_{i,j}}{L^\infty\bigl(J_i;L^\infty(\Omega)\bigr)}^2}$ and $C_{c_i}:=\Norm{c_i}{L^\infty\bigl(J_i;L^\infty(\Omega)\bigr)}$.
Then we have the following lemma.
\begin{lemma}\label{est-v=v-hat+v-0}
If $\lambda_{\min}>0$ then the solution $v$ of \Ref{para-res-v-newton} is estimated as follows:
\begin{eqnarray*}
\Norm{v(t_i)}{H^1_0(\Omega)}
&\le&\rho(T_i)\Norm{\epsilon_i}{H^1_0(\Omega)}+\sqrt{\frac{1}{\nu}}C_{\Delta_i}\rho_{\Omega}(T_i)\left(C_{b_i}\Norm{\epsilon_i}{H^1_0(\Omega)}+C_{c_i}\Norm{\epsilon_i}{L^2(\Omega)}\right),\\
\Norm{v(t_i)}{L^2(\Omega)}
&\le&\rho(T_i)\Norm{\epsilon_i}{L^2(\Omega)}+{\cal M}_{t_i}\rho_{\Omega}(T_i)\left(C_{b_i}\Norm{\epsilon_i}{H^1_0(\Omega)}+C_{c_i}\Norm{\epsilon_i}{L^2(\Omega)}\right),\\
\Norm{v}{L^2\bigl(J_i;H^1_0(\Omega)\bigr)}
&\le&\rho_{\Omega}(T_i)\Norm{\epsilon_i}{H^1_0(\Omega)}+{\cal M}_{1_i}\rho_{\Omega}(T_i)\left(C_{b_i}\Norm{\epsilon_i}{H^1_0(\Omega)}+C_{c_i}\Norm{\epsilon_i}{L^2(\Omega)}\right),\\
\Norm{v}{L^2\bigl(J_i;L^2(\Omega)\bigr)}
&\le&\rho_{\Omega}(T_i)\Norm{\epsilon_i}{L^2(\Omega)}+{\cal M}_{0_i}\rho_{\Omega}(T_i)\left(C_{b_i}\Norm{\epsilon_i}{H^1_0(\Omega)}+C_{c_i}\Norm{\epsilon_i}{L^2(\Omega)}\right).
\end{eqnarray*}
Moreover, we have
\begin{eqnarray*}
\Norm{v(t)}{H^1_0(\Omega)}
&\le&\Norm{\epsilon_i}{H^1_0(\Omega)}+\sqrt{\frac{1}{\nu}}C_{\Delta_i}\rho_{\Omega}(T_i)\left(C_{b_i}\Norm{\epsilon_i}{H^1_0(\Omega)}+C_{c_i}\Norm{\epsilon_i}{L^2(\Omega)}\right), \quad\forall t\in J_i,\\
\Norm{v}{V^1\bigl(J_i;L^2(\Omega)\bigr)}
&\le&\sqrt{\frac{\nu}{2}}\Norm{\epsilon_i}{H^1_0(\Omega)}+C_{\Delta_i}\rho_{\Omega}(T_i)\left(C_{b_i}\Norm{\epsilon_i}{H^1_0(\Omega)}+C_{c_i}\Norm{\epsilon_i}{L^2(\Omega)}\right).
\end{eqnarray*}
\end{lemma}
\begin{proof}
First, by the well known property of spectral theory, we have
\begin{eqnarray*}
\Norm{\hat{v}(t)}{H^1_0(\Omega)}
&\le&\exp(-\lambda_{\min}(t-t_{i-1}))\Norm{\epsilon_i}{H^1_0(\Omega)}\quad\forall t\in J_i,\\
\Norm{\hat{v}(t)}{L^2(\Omega)}
&\le&\exp(-\lambda_{\min}(t-t_{i-1}))\Norm{\epsilon_i}{L^2(\Omega)}\quad\forall t\in J_i.
\end{eqnarray*}
Also, noting that $\exp(-\lambda_{\min}(t-t_{i-1}))\le 1$ for $t\in J_i$ implies $\Norm{\hat{v}(t)}{H^1_0(\Omega)}\le\Norm{\epsilon_i}{H^1_0(\Omega)}$.
Thus, by integrating the above inequalities in $t$ on $J_i$, we have
\begin{eqnarray*}
\Norm{\hat{v}}{L^2\bigl(J_i;H^1_0(\Omega)\bigr)}^2
&=&\int_{J_i}\Norm{\hat{v}(t)}{H^1_0(\Omega)}^2dt
\le\Norm{\epsilon_i}{H^1_0(\Omega)}^2\int_{J_i}\exp(-2\lambda_{\min}(t-t_{i-1}))dt,\\
\Norm{\hat{v}}{L^2\bigl(J_i;L^2(\Omega)\bigr)}^2
&=&\int_{J_i}\Norm{\hat{v}(t)}{L^2(\Omega)}^2dt
\le\Norm{\epsilon_i}{L^2(\Omega)}^2\int_{J_i}\exp(-2\lambda_{\min}(t-t_{i-1}))dt,
\end{eqnarray*}
and
$$
\Norm{\hat{v}(t_i)}{H^1_0(\Omega)}
\le\rho(T_i)\Norm{\epsilon_i}{H^1_0(\Omega)},\quad 
\Norm{\hat{v}(t_i)}{L^2(\Omega)}
\le\rho(T_i)\Norm{\epsilon_i}{L^2(\Omega)}.
$$
Hence, by the simple computation, we obtain 
\begin{eqnarray}\label{est007}
\Norm{\hat{v}}{L^2\bigl(J_i;H^1_0(\Omega)\bigr)}
\le\rho_{\Omega}(T_i)\Norm{\epsilon_i}{H^1_0(\Omega)},\quad
\Norm{\hat{v}}{L^2\bigl(J_i;L^2(\Omega)\bigr)}
\le\rho_{\Omega}(T_i)\Norm{\epsilon_i}{L^2(\Omega)}.
\end{eqnarray}
Next, by the definition of $\hat{v}$, we have, for any $t\in J_i$,
\begin{eqnarray*}
0=\Inner{\Dev{}{t}\hat{v}(t)}{\Dev{}{t}\hat{v}(t)}{L^2(\Omega)} + \nu\Inner{\Delta\hat{v}(t)}{\Dev{}{t}\hat{v}(t)}{L^2(\Omega)}=\Norm{\Dev{}{t}\hat{v}(t)}{L^2(\Omega)}^2+\frac{\nu}{2}\frac{d}{dt}\Norm{\hat{v}(t)}{H^1_0(\Omega)}^2,
\end{eqnarray*}
which yields, by integrating in $t$ on $J_i$, 
\begin{eqnarray*}
\Norm{\hat{v}}{V^1\bigl(J_i;L^2(\Omega)\bigr)}^2 
+\frac{\nu}{2}\Norm{\hat{v}(t_{i})}{H^1_0(\Omega)}^2
&=&\frac{\nu}{2}\Norm{\epsilon_i}{H^1_0(\Omega)}^2.
\end{eqnarray*}

On the other hand, by applying the same arguments as in deriving the estimations \Ref{verification-condition}, we have the following estimates
\begin{eqnarray}\label{est-v-0}
\begin{array}{rclcl}
\Norm{v_0}{L^2\bigl(J_i;H^1_0(\Omega)\bigr)}
&\le&{\cal M}_{1_i}\Norm{-b_{i}\cdot\nabla\hat{v}-c_{i}\hat{v}}{L^2\bigl(J_i;L^2(\Omega)\bigr)},\\
\Norm{v_0}{L^2\bigl(J_i;L^2(\Omega)\bigr)}
&\le&{\cal M}_{0_i}\Norm{-b_{i}\cdot\nabla\hat{v}-c_{i}\hat{v}}{L^2\bigl(J_i;L^2(\Omega)\bigr)},\\
\Norm{v_0(t_i)}{L^2(\Omega)}
&\le&{\cal M}_{t_i}\Norm{-b_{i}\cdot\nabla\hat{v}-c_{i}\hat{v}}{L^2\bigl(J_i;L^2(\Omega)\bigr)},\\
\Norm{v_0(t)}{H^1_0(\Omega)}
&\le&\sqrt{\frac{1}{\nu}}C_{\Delta_i}\Norm{-b_{i}\cdot\nabla\hat{v}-c_{i}\hat{v}}{L^2\bigl(J_i;L^2(\Omega)\bigr)}, \quad\forall t\in J_i,\\
\Norm{v_0}{V^1\bigl(J_i;L^2(\Omega)\bigr)}
&\le&C_{\Delta_i}\Norm{-b_{i}\cdot\nabla\hat{v}-c_{i}\hat{v}}{L^2\bigl(J_i;L^2(\Omega)\bigr)}.
\end{array}
\end{eqnarray}
Also, from the estimates \Ref{est007}, it follows that 
\begin{eqnarray}\label{g(u-h)}
\begin{array}{rclcl}
\Norm{-b_{i}\cdot\nabla\hat{v}-c_{i}\hat{v}}{L^2\bigl(J_i;L^2(\Omega)\bigr)}
&\le&C_{b_i}\Norm{\hat{v}}{L^2\bigl(J_i;H^1_0(\Omega)\bigr)}+C_{c_i}\Norm{\hat{v}}{L^2\bigl(J_i;L^2(\Omega)\bigr)}\\
&\le&\rho_{\Omega}(T_i)\left(C_{b_i}\Norm{\epsilon_i}{H^1_0(\Omega)}+C_{c_i}\Norm{\epsilon_i}{L^2(\Omega)}\right).
\end{array}
\end{eqnarray}
Therefore, since $v=\hat{v}+v_0$, combining the above arguments on $\hat{v}$ with the estimates \Ref{est-v-0} on $v_0$ and \Ref{g(u-h)}, we obtain the desired conclusion of the lemma.
\end{proof}

Finally, in the case of one space dimension, we note that the $L^{\infty}$ estimates for $v$ can be obtained by using Lemma \ref{est-v=v-hat+v-0} and the result in \cite{Schultz 1973} (see p.8). For example, if $\Omega=(0,1)$, then we have:
\begin{eqnarray}\label{L-infty}
\begin{array}{rcl}
\Norm{v}{L^\infty\bigl(J_i;L^\infty(\Omega)\bigr)}
&=&{\rm ess} \sup_{t\in J_i}{\rm ess} \sup_{x\in\Omega}|v(x,t)|\\ 
&\le&\frac{1}{2}{\rm ess} \sup_{t\in J_i}\Norm{v(t)}{H^1_0(\Omega)}\\
&\le&\frac{1}{2}\left(\Norm{\epsilon_i}{H^1_0(\Omega)}+\sqrt{\frac{1}{\nu}}C_{\Delta_i}\rho_{\Omega}(T_i)\left(C_{b_i}\Norm{\epsilon_i}{H^1_0(\Omega)}+C_{c_i}\Norm{\epsilon_i}{L^2(\Omega)}\right)\right).
\end{array}
\end{eqnarray}
\subsection{Some remarks on the estimation for nonlinear terms in $g_i(w)$}
\label{w-g_i(w)}

Since some nonlinear terms in $v$ and $w$ appear in the right-hand side of the problem \Ref{fixed-point}, in order to validate the verification condition \Ref{condition}, we need several kinds of techniques to estimate them. In what follows, we only consider for one dimensional case, i.e., $\Omega=(0,1)=:\Omega_1$ and, as an example of a typical nonlinear term, we show how to estimate the power of $v$ or $w$. 

First, observe that the inequality for any $v \in L^{\infty}\bigl(J_i;H^1_0(\Omega_1)\bigr)$
\begin{eqnarray}\label{v^p-est}
\Norm{v^p}{L^2\bigl(J_i;L^2(\Omega_1)\bigr)}&\leq & \Norm{v}{L^\infty\bigl(J_i;L^\infty(\Omega_1)\bigr)}^{p-1}\Norm{v}{L^2\bigl(J_i;L^2(\Omega_1)\bigr)},
\end{eqnarray}
which yields the desired estimates by using the results in Lemma \ref{est-v=v-hat+v-0} and \Ref{L-infty}.

Next, we estimate $\Norm{w^p}{L^2\bigl(J_i;L^2(\Omega_1)\bigr)}$ as below.
For any $w\in W_{\alpha_i,\beta_i}(\Omega_1,J_i)$, since $w(x,t_{i-1})=0$, we extend it to a function $\hat{w}$ on $\hat{J_i} \equiv (t_{i-1},2t_i-t_{i-1})$ satisfying the symmetry with respect to $t=t_{i}$. Then, noting that $\hat{w}\in H^1_0(\Omega_1\times \hat{J}_i)$, we apply the embedding theorem (e.g.\cite{Talenti}) to $\hat{w}$ to obtain the following estimates:
$$
\Norm{\hat{w}^p}{L^2(\Omega_1\times \hat{J}_i)}
\le
\sqrt{2T_i}K_w(p)^p\Norm{\hat{w}}{H^1_0(\Omega_1\times \hat{J}_i)}^p
\quad\mbox{where}\quad
K_w(p)\equiv \frac{p}{2\pi}(p-1)^{-\frac{1}{2p}}\left(\sin\frac{\pi}{p}\right)^\frac{1}{2},\quad p>1.
$$
It implies that
\begin{eqnarray*}
\Norm{w^p}{L^2\bigl(J_i;L^2(\Omega_1)\bigr)}^2
&=&\frac{1}{2}\Norm{\hat{w}^p}{L^2(\Omega_1\times \hat{J}_i)}^2\\
&\le&T_iK_w(p)^{2p}\Norm{\hat{w}}{H^1_0(\Omega_1\times \hat{J}_i)}^{2p}\\
&=&2^{p}T_iK(p)^{2p}\left(\Norm{w}{V^1\bigl(J_i;L^2(\Omega_1)\bigr)}^2+\Norm{w}{L^2\bigl(J_i;H^1_0(\Omega_1)\bigr)}^2\right)^{p}.
\end{eqnarray*}
Therefore, we have
\begin{eqnarray}\label{embedding}
\Norm{w^p}{L^2\bigl(J_i;L^2(\Omega_1)\bigr)}
\le\tilde{K}_w(p)\left(\Norm{w}{V^1\bigl(J_i;L^2(\Omega_1)\bigr)}^2+\Norm{w}{L^2\bigl(J_i;H^1_0(\Omega_1)\bigr)}^2\right)^{\frac{p}{2}},\quad p>1,
\end{eqnarray}
where $\tilde{K}_w(p)\equiv \sqrt{2^{p}T_i}K_w(p)^{p}$.
Note that $\tilde{K}_w(2)<0.2027\sqrt{T_i}$ and $\tilde{K}_w(3)<0.1755\sqrt{T_i}$.

We now show some examples of the estimation for later use based on the above discussion for the problem \Ref{para-res} with quadratic and cubic nonlinearities in $\bar{u}$. Let $v$ be a solution of \Ref{para-res-v-newton}, and let $w$ be an element in a candidate set of the problem \Ref{para-res-w-newton}, i.e., $w\in W_{\alpha_i,\beta_i}(\Omega_1,J_i)$. Then we have the following estimates
\begin{eqnarray}
\Norm{(v+w)^2}{L^2\bigl(J_i;L^2(\Omega_1)\bigr)}
&\le&
\Norm{v^2}{L^2\bigl(J_i;L^2(\Omega_1)\bigr)}
+\Norm{w^2}{L^2\bigl(J_i;L^2(\Omega_1)\bigr)}
+2\Norm{vw}{L^2\bigl(J_i;L^2(\Omega_1)\bigr)} \nonumber \\
&\le&
\Norm{v}{L^\infty\bigl(J_i;L^\infty(\Omega_1)\bigr)}\Norm{v}{L^2\bigl(J_i;L^2(\Omega_1)\bigr)}
+\tilde{K}_w(2)\left(\alpha_i^2+\beta_i^2\right) \nonumber \\
& & \hspace*{5cm}+C_p\Norm{v}{L^\infty\bigl(J_i;L^\infty(\Omega_1)\bigr)}\alpha_i \nonumber \\
&=:&G_2(\alpha_i,\beta_i) \label{G_2}
\end{eqnarray}
and
\begin{eqnarray}
\Norm{(v+w)^3}{L^2\bigl(J_i;L^2(\Omega_1)\bigr)}
&\le&
\Norm{v^3}{L^2\bigl(J_i;L^2(\Omega_1)\bigr)}
+\Norm{w^3}{L^2\bigl(J_i;L^2(\Omega_1)\bigr)} \nonumber \\
&&
+3\Norm{vw^2}{L^2\bigl(J_i;L^2(\Omega_1)\bigr)}
+3\Norm{v^2w}{L^2\bigl(J_i;L^2(\Omega_1)\bigr)} \nonumber \\
&\le&
\Norm{v}{L^\infty\bigl(J_i;L^\infty(\Omega_1)\bigr)}^2\Norm{v}{L^2\bigl(J_i;L^2(\Omega_1)\bigr)}
+\tilde{K}_w(3)\left(\alpha_i^2+\beta_i^2\right)^{\frac{3}{2}} \nonumber \\
&&
+3\Norm{v}{L^\infty\bigl(J_i;L^\infty(\Omega_1)\bigr)}\tilde{K}_w(2)\left(\alpha_i^2+\beta_i^2\right)
+3C_p\Norm{v}{L^\infty\bigl(J_i;L^\infty(\Omega_1)\bigr)}^2\alpha_i \nonumber \\
&=:&G_3(\alpha_i,\beta_i).\label{G_3}
\end{eqnarray}
Here $C_p>0$ is the Poincar\'e constant. (It is taken as $C_p=1/\pi$ in the case of $\Omega_1$.)
\setcounter{equation}{0}
\section{Numerical examples}
In this section, we show some examples whose solutions are verified by our method and, as in Section \ref{w-g_i(w)}, we only consider for one dimensional case, i.e., $\Omega=(0,1)=:\Omega_1$. \\
First, we describe several remarks on the verification step from the interval $J_i$ to $J_{i+1}$. Let $\alpha_i^*$ and $\beta_i^*$ be two positive numbers satisfying the condition \Ref{condition}. Then there exists a solution $w_i^*\in W_{\alpha_i^*,\beta_i^*}(\Omega_1,J_i)$ of the problem \Ref{para-res-w-newton} and the following estimates hold 
$$
\Norm{w_i^*(t_i)}{L^2(\Omega_1)}\le{\cal M}_{t_i}G(\alpha_i^*,\beta_i^*),\quad
\Norm{w_i^*(t_i)}{H^1_0(\Omega_1)}\le\sqrt{\frac{1}{\nu}}C_{\Delta_i}G(\alpha_i^*,\beta_i^*).
$$
When denoting $v_i^*$ as a solution of the problem \Ref{para-res-v-newton}, the solution $u_i^*$ of the nonlinear problem \Ref{para-nonlinear} on $\Omega_1\times J_i$ can be written by $u_i^*=\bar{u}_{h,i}^{k}+v_i^*+w_i^*$. 
Note that the initial condition of the next time-step problem in $\Omega_1\times J_{i+1}$ is given by $u_i^*(t_i)=\bar{u}_{h,i}^{k}(t_i)+v_i^*(t_i)+w_i^*(t_i)$ .
Since we take the approximate solution $\bar{u}_{h,i+1}\in \bar{S}_h^k(\Omega_1,J_{i+1})$ satisfying $\bar{u}_{h,i+1}(t_i)=\bar{u}_{h,i}^{k}(t_i)$, an initial function $\epsilon_{i+1}$ of the problem \Ref{para-res} in $\Omega_1\times J_{i+1}$ is given by 
$$
\epsilon_{i+1}:=v_i^*(t_i)+w_i^*(t_i),\quad i=1,\cdots,l-1.
$$
Therefore, we can obtain the following estimations:
\begin{eqnarray*}
\Norm{\epsilon_{i+1}}{H^1_0(\Omega_1)}
&\le&
\Norm{v_i^*(t_i)}{H^1_0(\Omega_1)}
+\sqrt{\frac{1}{\nu}}C_{\Delta_i}G(\alpha_i^*,\beta_i^*), \\
\Norm{\epsilon_{i+1}}{L^2(\Omega_1)}
&\le&
\Norm{v_i^*(t_i)}{L^2(\Omega_1)}
+{\cal M}_{t_i}G(\alpha_i^*,\beta_i^*).
\end{eqnarray*}
In the following examples, we take the basis of finite element subspaces $S_h(\Omega_1)$ and $V_k^1(J_i)$ as the piecewise linear (P1) function. On the other hand, for computing approximations $\bar{u}_{h,i}^{k}$ of nonlinear problems, we take the basis of finite element subspaces $\bar{S}_h(\Omega_1)$ and $\bar{V}_k^1(J_i)$ as the piecewise Hermite spline ($C^1$-class with 5-degree) function and the piecewise quadratic ($C^0$-class) function, respectively. 

\begin{example}\label{ex:Fujita}
Fujita-type equation:
$$
\nu = 1 \quad \mbox{and} \quad g(u):=u^2.
$$
We take the initial function $u_0$ as $u_0=32x(x-1)(x^2-x-1)$, and consider the problem for $\Omega_1=(0,1)$. ($u_0(1/2)=10$)
\end{example}

For the example \ref{ex:Fujita}, the linearized part $-g^{\prime}[\bar{u}_{h,i}^{k}]w$ is given by $-g^{\prime}[\bar{u}_{h,i}^{k}]w=-2\bar{u}_{h,i}^{k}w$ for $w\in W_{\alpha_i,\beta_i}(\Omega_1,J_i)$ then we obtain coefficient functions as $b_i=0$ and $c_i=-2\bar{u}_{h,i}^{k}$.
Moreover, it follows that
\begin{eqnarray*}
\Norm{g_i(w)}{L^2\bigl(J_i;L^2(\Omega_1)\bigr)}
&=&
\Norm{\delta_i+(v+w)^2}{L^2\bigl(J_i;L^2(\Omega_1)\bigr)}\\
&\le&
\Norm{\delta_i}{L^2\bigl(J_i;L^2(\Omega_1)\bigr)}
+G_2(\alpha_i,\beta_i),
\end{eqnarray*}
where $G_2(\alpha_i,\beta_i)$ is defined by \Ref{G_2}. 
Hence we define the function $G(\alpha_i,\beta_i)$ in the verification condition \Ref{G} by
\begin{eqnarray*}
G(\alpha_i,\beta_i)
&:=&
\Norm{\delta_i}{L^2\bigl(J_i;L^2(\Omega_1)\bigr)}
+G_2(\alpha_i,\beta_i).
\end{eqnarray*}
\begin{example}\label{ex:Allen}
Allen-Cahn equation:
$$
\nu = \frac{1}{150} \quad \mbox{and} \quad g(u):=u(1-u)(u-a),
$$
where $a>0$ is a constant. 
We take the initial function $u_0$ as $u_0=x(x-1)(x^2-x-1)$ and $a=0.01$, and consider the problem for $\Omega_1=(0,1)$. ($u_0(1/2)=0.3125$)
\end{example}

For the example \ref{ex:Allen}, the linearized part $-g^{\prime}[\bar{u}_{h,i}^{k}]w$ is given by $-g^{\prime}[\bar{u}_{h,i}^{k}]w=(a-2(1+a)\bar{u}_{h,i}^{k}+3(\bar{u}_{h,i}^{k})^2)w$ for $w\in W_{\alpha_i,\beta_i}(\Omega_1,J_i)$ then we obtain coefficient functions as $b_i=0$ and $c_i=a-2(1+a)\bar{u}_{h,i}^{k}+3(\bar{u}_{h,i}^{k})^2$.
Moreover, it follows that
\begin{eqnarray*}
\Norm{g_i(w)}{L^2\bigl(J_i;L^2(\Omega_1)\bigr)}
&=&
\Norm{\delta_i+d_i(v+w)^2+(v+w)^3}{L^2\bigl(J_i;L^2(\Omega_1)\bigr)}\\
&\le&
\Norm{\delta_i}{L^2\bigl(J_i;L^2(\Omega_1)\bigr)}
+C_{d_i}G_2(\alpha_i,\beta_i)
+G_3(\alpha_i,\beta_i),
\end{eqnarray*}
where $d_i=1+a-3\bar{u}_{h,i}^{k}$, $C_{d_i}:=\Norm{d_i}{L^\infty\bigl(J_i;L^\infty(\Omega_1)\bigr)}$ and $G_3(\alpha_i,\beta_i)$ is defined by \Ref{G_3}.
Thus we define the function $G(\alpha_i,\beta_i)$ in \Ref{G} by
\begin{eqnarray*}
G(\alpha_i,\beta_i)
&:=&
\Norm{\delta_i}{L^2\bigl(J_i;L^2(\Omega_1)\bigr)}
+C_{d_i}G_2(\alpha_i,\beta_i)
+G_3(\alpha_i,\beta_i).
\end{eqnarray*}
\begin{figure}
\begin{center}
\includegraphics[width=10cm,height=5cm]{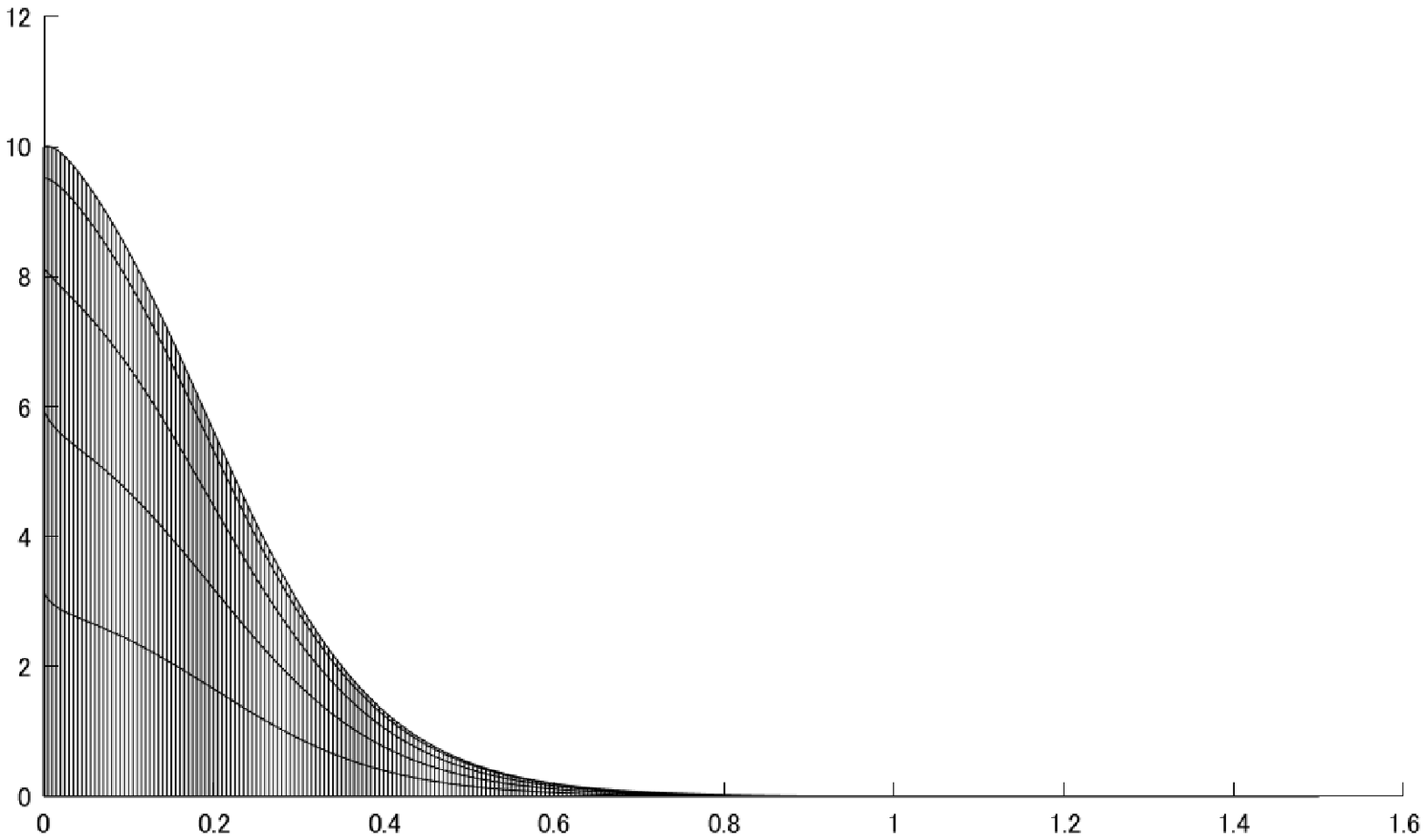}
\includegraphics[width=5cm,height=5cm]{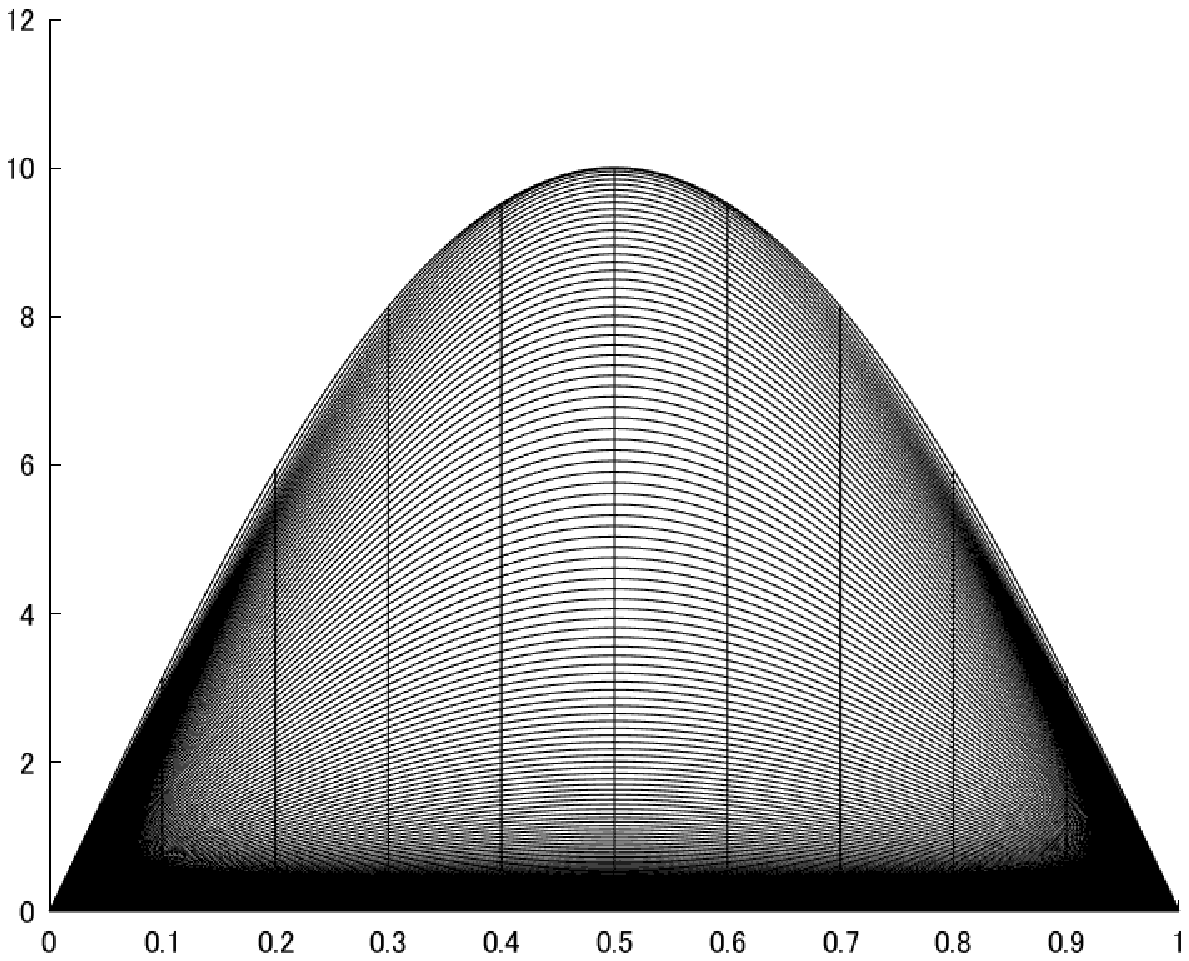}
\caption{Approximations $\bar{u}_{h,i}^{k}$ for Example \ref{ex:Fujita} ($\nu=1$)}\label{fig:Fujita}
\includegraphics[width=10cm,height=5cm]{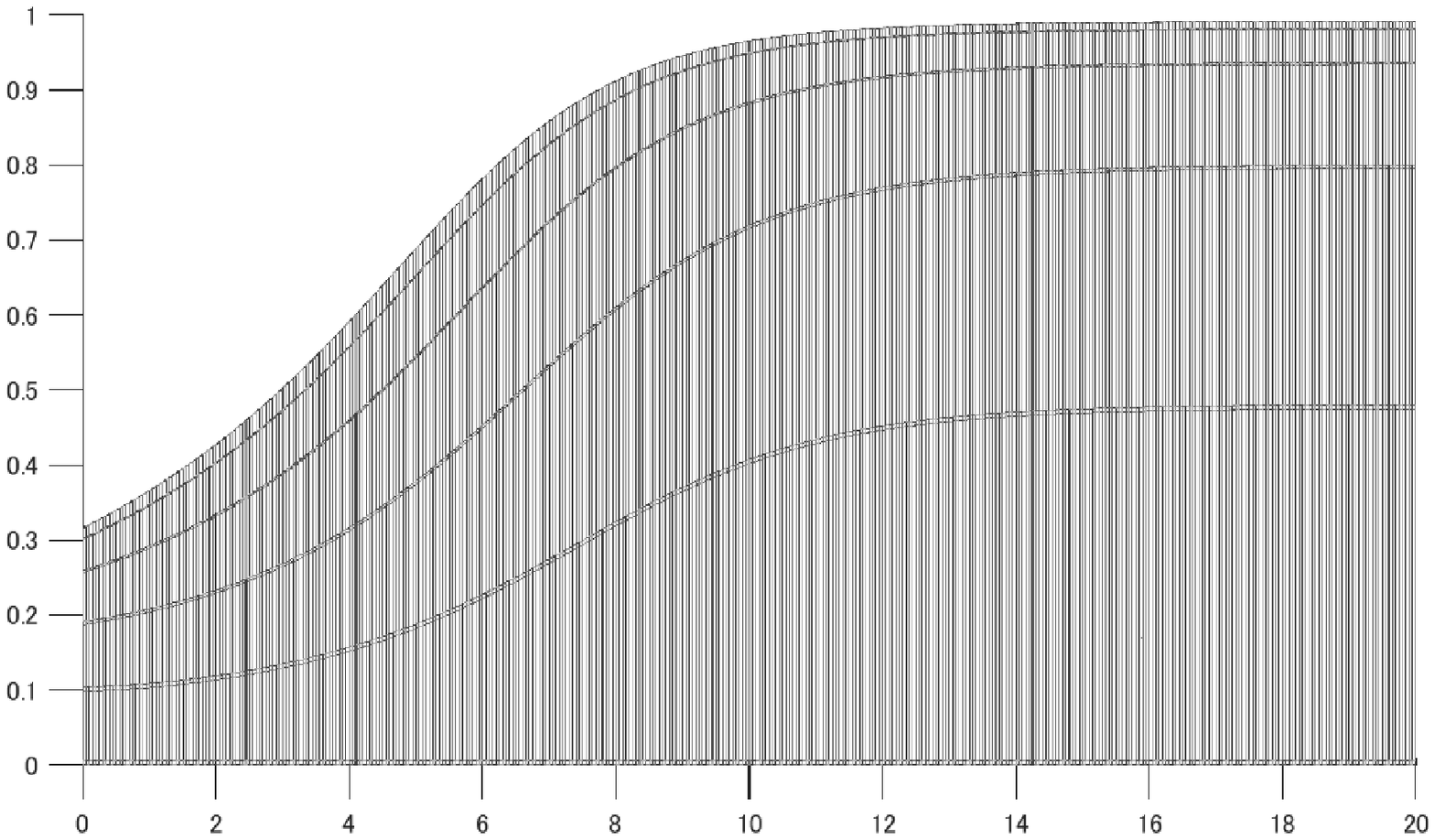}
\includegraphics[width=5cm,height=5cm]{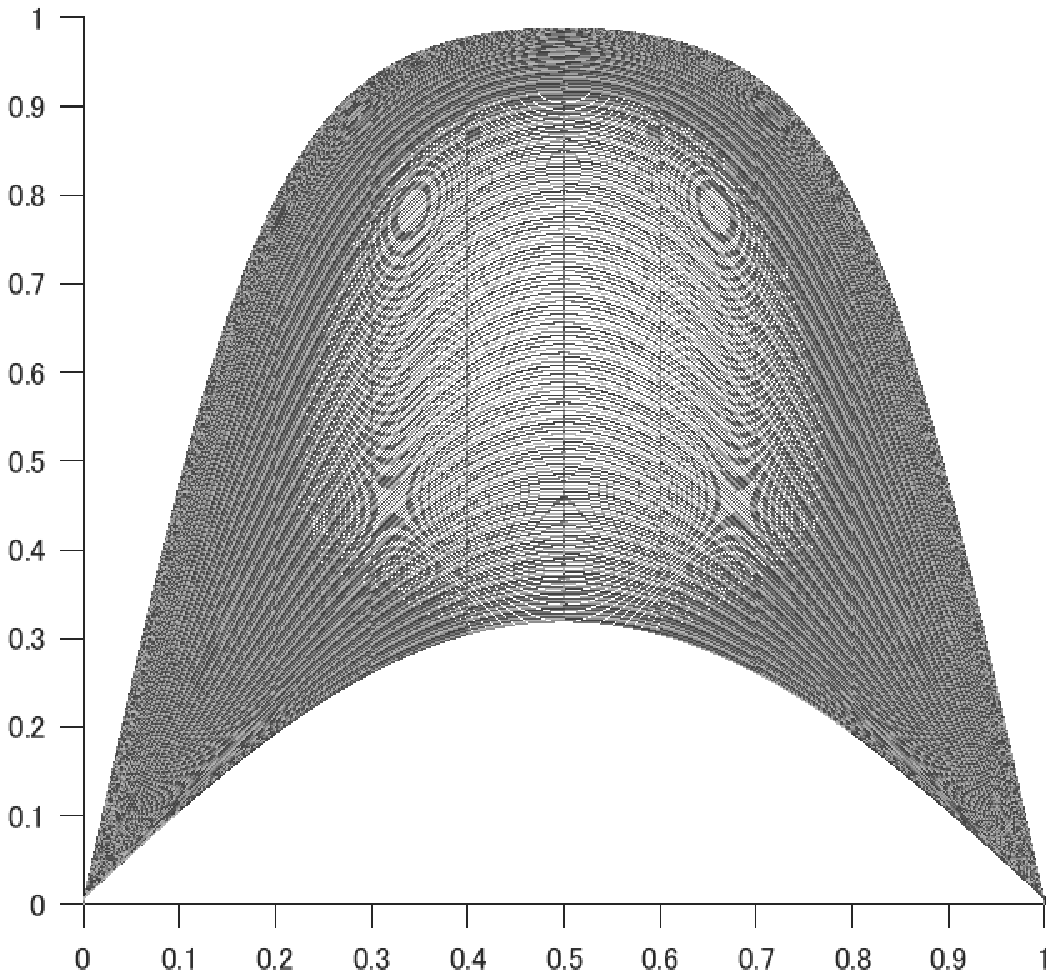}
\caption{Approximations $\bar{u}_{h,i}^{k}$ for Example \ref{ex:Allen} ($\nu=1/150$)}\label{fig:Allen}
\end{center}
\end{figure}
\begin{table}
\begin{center}
\caption{Verified numerical results for Example \ref{ex:Fujita} $T_i=0.1$, $(h,k)=(1/10,1/1000)$ }
\footnotesize
\begin{tabular}{|c||c|c|c|c|c|c|c|c|c|c||}
\hline
$i$&${\cal M}_{1_i}$&${\cal M}_{0_i}$&${\cal M}_{t_i}$&$C_{\Delta_i}$&$M_{1}$&$M_{0}$&$M_{T}$&$\alpha_i$&$\beta_i$&$\Norm{\delta_i}{L^2\bigl(J_i;L^2(\Omega_1)\bigr)}$\\
\hline\hline
1&1.035&0.230&1.452&5.616&0.261&0.082&0.426&9.31E-04&5.06E-03&8.90E-04\\
2&0.632&0.142&0.856&3.379&0.219&0.069&0.336&3.10E-03&1.65E-02&3.02E-04\\
3&0.393&0.090&0.524&2.011&0.180&0.057&0.267&4.40E-03&2.25E-02&1.64E-04\\
4&0.291&0.068&0.388&1.405&0.157&0.050&0.231&2.81E-03&1.35E-02&9.53E-05\\
5&0.250&0.059&0.337&1.156&0.147&0.046&0.217&9.22E-04&4.27E-03&6.06E-05\\
6&0.235&0.056&0.317&1.059&0.142&0.045&0.211&2.01E-04&9.13E-04&3.81E-05\\
7&0.229&0.055&0.310&1.022&0.141&0.045&0.209&3.74E-05&1.68E-04&2.36E-05\\
8&0.227&0.055&0.308&1.008&0.140&0.044&0.209&8.09E-06&3.60E-05&1.45E-05\\
9&0.226&0.054&0.307&1.003&0.140&0.044&0.208&2.71E-06&1.20E-05&8.84E-06\\
10&0.226&0.054&0.306&1.001&0.140&0.044&0.208&1.32E-06&5.88E-06&5.40E-06\\
$\vdots$&$\vdots$&$\vdots$&$\vdots$&$\vdots$&$\vdots$&$\vdots$&$\vdots$&$\vdots$&$\vdots$&$\vdots$\\
15&0.225&0.054&0.306&1.000&0.140&0.044&0.208&1.04E-07&4.61E-07&4.58E-07\\
$\vdots$&$\vdots$&$\vdots$&$\vdots$&$\vdots$&$\vdots$&$\vdots$&$\vdots$&$\vdots$&$\vdots$&$\vdots$\\
20&0.225&0.054&0.306&1.000&0.140&0.044&0.208&8.78E-09&3.91E-08&3.88E-08\\
30&0.225&0.054&0.306&1.000&0.140&0.044&0.208&6.25E-11&2.78E-10&2.77E-10\\
40&0.225&0.054&0.306&1.000&0.140&0.044&0.208&8.63E-12&3.80E-11&3.80E-11\\
50&0.225&0.054&0.306&1.000&0.140&0.044&0.208&8.63E-12&3.80E-11&3.80E-11\\
\multicolumn{11}{|c||}{Stop}\\
\hline
\end{tabular}
$\tilde{C}_1(h,k)=0.0857$ ($\gamma_1=0.999$),\quad$\tilde{C}_0(h,k)=0.0099$ ($\gamma_0=0.139$),\quad$\tilde{c}_0(h,k)=0.0978$ ($\gamma_T=0.707$). 
\label{table3-Fujita} 
\end{center}
\end{table} 
\begin{table}
\begin{center}
\caption{Verified numerical results for Example \ref{ex:Allen} $T_i=1$, $(h,k)=(1/64,1/128)$ }
\footnotesize
\begin{tabular}{|c||c|c|c|c|c|c|c|c|c|c||}
\hline
$i$&${\cal M}_{1_i}$&${\cal M}_{0_i}$&${\cal M}_{t_i}$&$C_{\Delta_i}$&$M_{1}$&$M_{0}$&$M_{T}$&$\alpha_i$&$\beta_i$&$\Norm{\delta_i}{L^2\bigl(J_i;L^2(\Omega_1)\bigr)}$\\
\hline\hline
1&9.581&0.788&1.420&1.260&6.175&0.706&1.138&3.75E-07&4.96E-08&3.90E-08\\
2&9.591&0.789&1.420&1.261&6.183&0.706&1.138&3.46E-07&4.54E-08&1.80E-08\\
3&9.577&0.785&1.410&1.259&6.173&0.702&1.128&5.63E-07&7.39E-08&2.59E-08\\
4&9.527&0.773&1.382&1.255&6.134&0.691&1.102&8.71E-07&1.15E-07&3.39E-08\\
5&9.442&0.751&1.336&1.248&6.070&0.670&1.058&1.37E-06&1.81E-07&4.99E-08\\
6&9.348&0.723&1.278&1.239&6.002&0.643&1.004&2.03E-06&2.70E-07&6.73E-08\\
7&9.619&0.707&1.255&1.336&5.951&0.615&0.950&4.00E-06&5.57E-07&8.60E-08\\
8&9.984&0.697&1.245&1.453&5.915&0.590&0.903&8.42E-06&1.23E-06&8.82E-08\\
9&10.213&0.685&1.230&1.527&5.884&0.568&0.865&1.76E-05&2.63E-06&1.04E-07\\
10&10.330&0.673&1.214&1.567&5.855&0.551&0.836&3.60E-05&5.44E-06&1.12E-07\\
11&10.380&0.663&1.200&1.587&5.829&0.539&0.816&7.30E-05&1.13E-05&1.07E-07\\
12&10.399&0.656&1.190&1.596&5.809&0.531&0.802&1.48E-04&2.28E-05&1.12E-07\\
13&10.405&0.652&1.183&1.601&5.795&0.526&0.794&3.00E-04&4.62E-05&1.25E-07\\
14&10.406&0.649&1.179&1.604&5.786&0.522&0.789&6.02E-04&9.27E-05&1.12E-07\\
15&10.406&0.647&1.176&1.605&5.781&0.520&0.786&1.22E-03&1.90E-04&1.22E-07\\
16&10.406&0.646&1.175&1.606&5.777&0.519&0.784&2.50E-03&3.85E-04&1.13E-07\\
17&10.405&0.645&1.174&1.606&5.775&0.519&0.783&5.21E-03&8.03E-04&1.22E-07\\
18&10.405&0.645&1.173&1.606&5.774&0.518&0.783&1.14E-02&1.77E-03&1.13E-07\\
19&10.405&0.645&1.173&1.606&5.773&0.518&0.782&2.83E-02&4.36E-03&1.24E-07\\
20&10.405&0.645&1.173&1.607&5.773&0.518&0.782&Inf&Inf&1.20E-07\\
\hline
\end{tabular}
$\tilde{C}_1(h,k)=2.594$ ($\gamma_1=0.999$),\quad$\tilde{C}_0(h,k)=0.053$ ($\gamma_0=0.038$),\quad$\tilde{c}_0(h,k)=0.204$ ($\gamma_T=0.057$). 
\label{table3-Allen} 
\end{center}
\end{table} 
\begin{figure}
\begin{center}
\includegraphics[width=7cm,height=7cm]{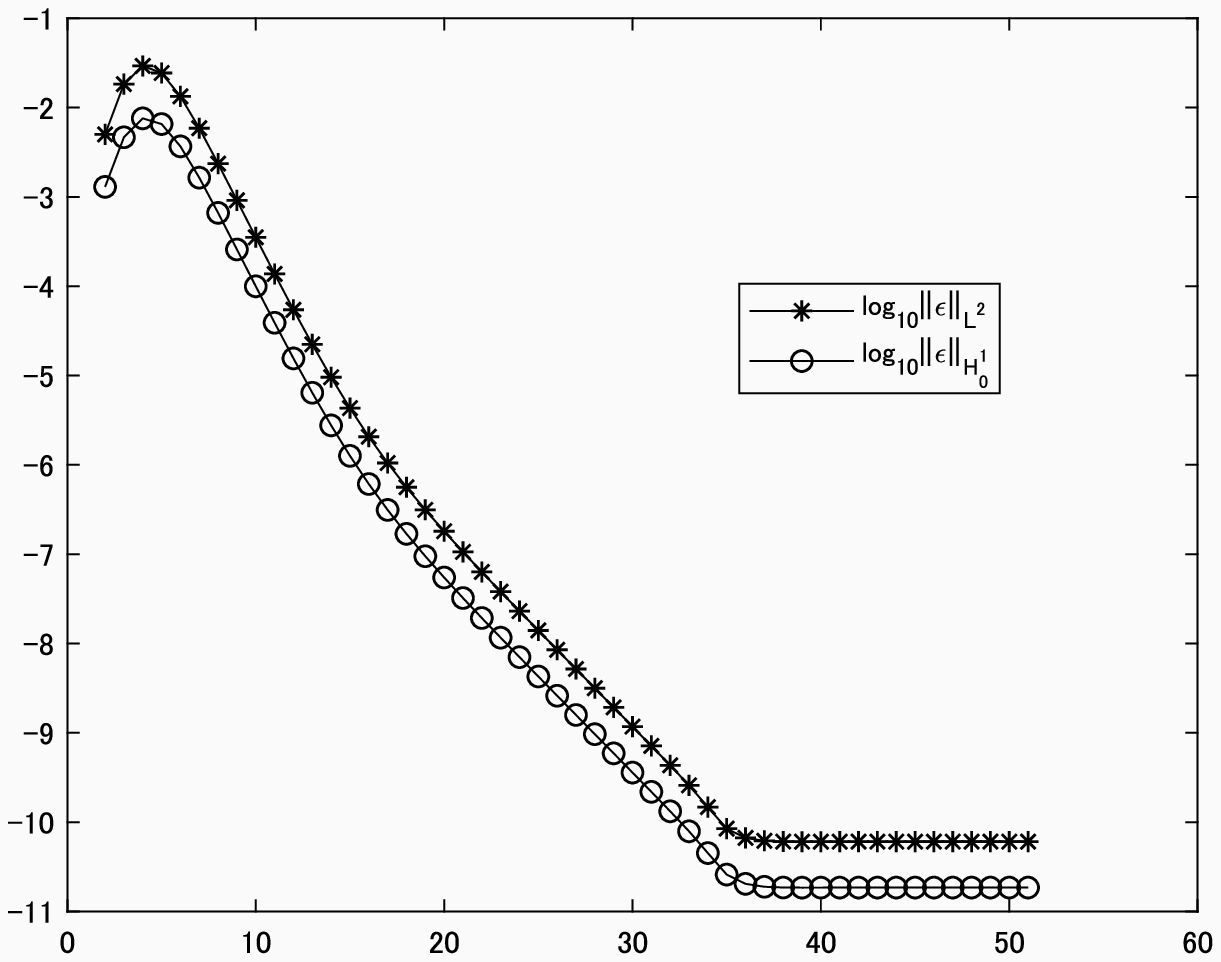}
\includegraphics[width=7cm,height=7cm]{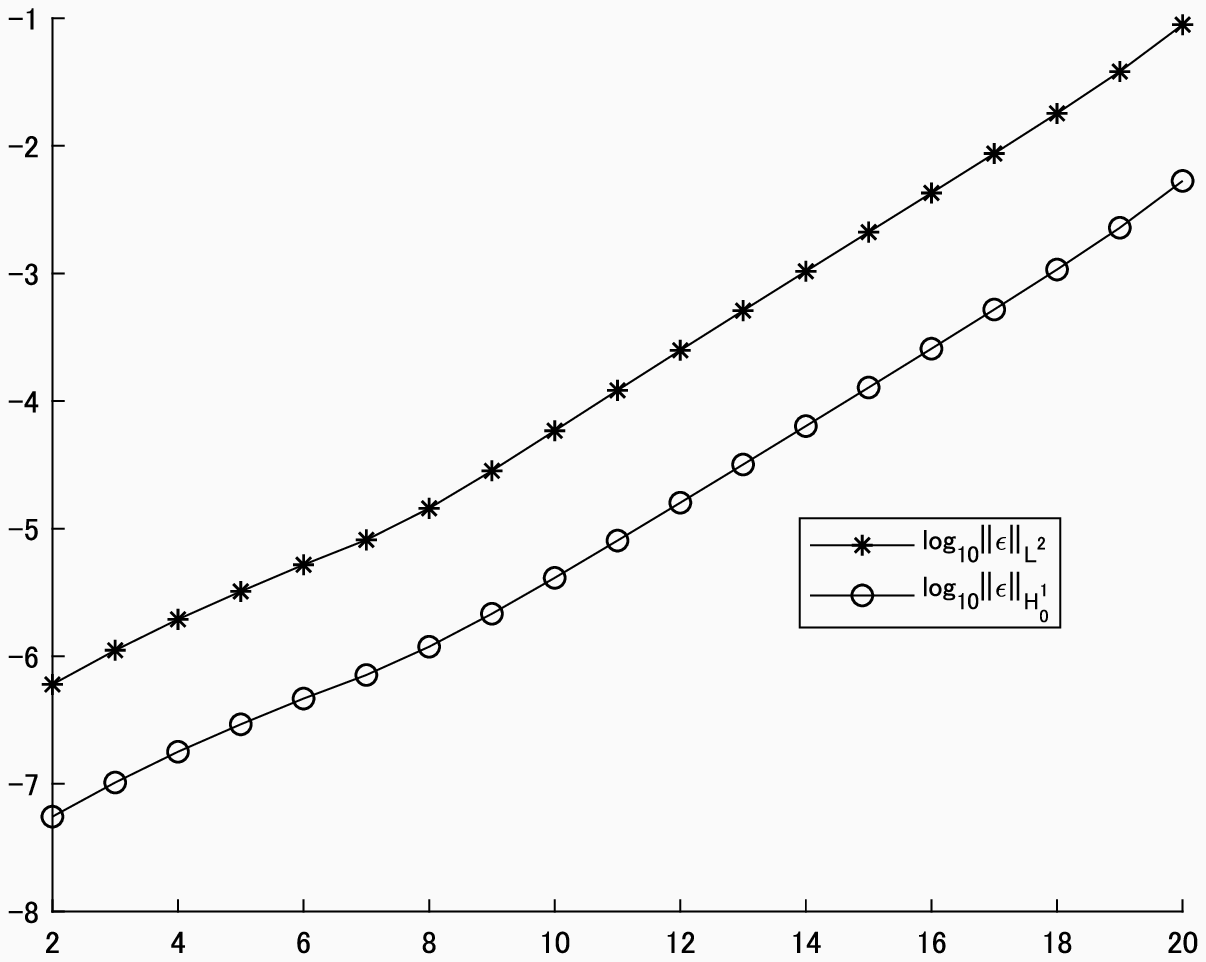}
\caption{$\log_{10}(\Norm{\epsilon_{i}}{L^2(\Omega_1)})$ and $\log_{10}(\Norm{\epsilon_{i}}{H^1_0(\Omega_1)})$ for Example \ref{ex:Fujita}(left) and Example \ref{ex:Allen}(right)}\label{fig:error}
\end{center}
\end{figure}

Verified computational results are shown in Table \ref{table3-Fujita}, Table \ref{table3-Allen} and Figure \ref{fig:error} as well as the approximate contours are illustrated in Figure \ref{fig:Fujita} and Figure \ref{fig:Allen}, respectively. In the left-hand sides of figures \ref{fig:Fujita} and \ref{fig:Allen}, the horizontal and vertical lines indicate the numbers of time-step and the size of the norms, respectively. And in the right-hand sides of these figures, these two directions imply the spatial coordinate axes and norms, respectively. Also we take uniform time-step size as $T_i=0.1$ and $T_i=1$ for Example \ref{ex:Fujita} and Example \ref{ex:Allen}, respectively. 
We compute approximate solutions $\bar{u}_{h,i}^{k}$ of Example \ref{ex:Fujita} and Example \ref{ex:Allen} by the double precision. Particularly, Figure \ref{fig:error} shows the accumulated error behavior with time progression for each time step (transverse line) $t_i$.
Hence it can be deduced the solution of Example 1 rapidly decreases as time increases.

{\bf Remark}:\ 
All computations in Tables are carried out on the Dell Precision 7920 Intel Xeon Gold 6134 CPU 3.20GHz by using INTLAB(var. 10.1), a tool box in MATLAB (var. R2018a) developed by Rump \cite{Rump INTLAB} for self-validating algorithms. Therefore, all numerical values in these tables are verified data in the sense of strictly rounding error control. Also we used Symbolic Math Toolbox for $\delta_i$, for the norm estimation of linearized operator $\mathcal{L}_{i}$, 
\setcounter{equation}{0}
\section{Conclusion}
We presented a numerical verification method of solutions for nonlinear parabolic initial boundary value problems.
Using our method, we showed numerically verified results for a solution from the initial value to the neighborhood of the stationary solution of Fujita-type equation and Allen-Cahn equation.
For the solution which decays to zero of Fujita-type equation, we succeeded in the verification without any accumulation of the error at each time step $t_i$, which suggests that for such kind of problems  our present approach should be really effective. On the other hand, 
in case of Allen-Cahn equation, the error actually accumulate at each time step. In such a case it is shown, by our numerical results, that the application of the theoretical analysis of the heat equation to estimate of the initial part $v$ should be more effective.
In conclusion, we can say that our method is the first effective approach to the numerical verification of solutions for actually realistic nonlinear evolution equations based on the finite element method by using a  Newton-type formulation.\\ \  \\
{\bf Acknowledgement}:\ 
This work was partially supported by JSPS KAKENHI Grant Number 18K03434, 18K03440 and JST CREST.


\end{document}